\documentclass[a4paper,11pt,leqno]{article}
\usepackage[latin1]{inputenc}
\usepackage[T1]{fontenc} 
\usepackage[english]{babel}
\usepackage{verbatim}
\usepackage{calligra}
\usepackage{enumerate}
\usepackage{dsfont}
\usepackage{amsfonts}
\usepackage{amsmath}
\usepackage{amsthm}
\usepackage{amssymb}
\usepackage{esint}
\usepackage{soul}
\usepackage{mathtools}
\usepackage{mathrsfs}
\usepackage{color}
\usepackage{graphicx}
\usepackage{bm}
  \usepackage[all]{xy}
\usepackage{hyperref}
\hypersetup{
    colorlinks=true,
    linkcolor=black,
    filecolor=black,      
    urlcolor=black,
}

\usepackage{tikz}

\usepackage{graphicx}		
\usepackage[hypcap=false]{caption}

\DeclareMathOperator*{\tend}{\longrightarrow}

\DeclareMathOperator*{\D}{\rm{div}}

\DeclareMathOperator*{\wtend}{\rightharpoonup}
\DeclareMathOperator*{\limss}{\overline{\rm lim}}

\theoremstyle{definition}
\newtheorem{defi}{Definition}

\newtheorem{rmk}[defi]{Remark}

\theoremstyle{plain}
\newtheorem{thm}[defi]{Theorem}
\newtheorem{prop}[defi]{Proposition}

\newtheorem{lemma}[defi]{Lemma}

\newcommand{\mc}{\mathcal}

\newcommand{\what}{\widehat}

\newcommand{\s}{\sigma}

\newcommand{\R}{\mathbb{R}}

\newcommand{\N}{\mathbb{N}}
\newcommand{\Z}{\mathbb{Z}}
\newcommand{\T}{\mathbb{T}}
\renewcommand{\P}{\mathbb{P}}

\newcommand{\dx}{ \, {\rm d} x}
\newcommand{\dt}{ \, {\rm d} t}

\allowdisplaybreaks

\def\avint{\mathop{\,\rlap{-}\!\!\int}\nolimits}

\begin{document}

\newcommand{\cobb}[1]{\textcolor{red}{[***DC: #1 ***]}}
\newcommand{\lacour}[1]{\textcolor{blue}{[***GL: #1 ***]}}

\title{\textsc{\Large{\textbf{Weak Solutions for a non-Newtonian Stokes-Transport System}}}}

\author{Dimitri \textsc{Cobb}\footnote{Universität Bonn, \textit{Mathematisches Institut}, Endenicher Allee 60, 53115 Bonn, Germany.\\
\hspace*{0.5cm} \texttt{cobb@math.uni-bonn.de}}$\,$ and Geoffrey \textsc{Lacour}\footnote{Université Clermont Auvergne, CNRS, LMBP, F-63000 Clermont-Ferrand, France.\\ \hspace*{0.5cm} \texttt{geoffrey.lacour@uca.fr}}}

\vspace{.2cm}
\date\today

\maketitle

\begin{abstract}
    In this article, we study a non-Newtonian Stokes-Transport system. This set of PDEs was introduced as a model for describing the behavior of a cloud of particles in suspension in a Stokes fluid, and is a nonlinear coupling between a hyperbolic equation (Transport) and a nonlinear elliptic equation (non-Newtonian Stokes), and as such can be considered as an active scalar equation. We prove the existence of global weak solutions with initial data in critical Lebesgue spaces. In order to overcome the difficulties introduced by the highly nonlinear aspect of this problem, we resort to a combination of DiPerna-Lions theory of transport equations and Minty's trick for elliptic equations.
\end{abstract}


\medskip


\section{Introduction}

The purpose of this paper is to study the existence of global weak solutions of the following Stokes-Transport system with non-Newtonian viscosity:
\begin{equation}\label{ieq:ST}
    \begin{cases}
    \partial_t \rho + u \cdot \nabla \rho = 0 \\
    - \D \left( \nu(\rho) |Du|^{p-2} Du \right) + \nabla \pi = \rho g\\
    \D(u) = 0.
    \end{cases}
\end{equation}
In the equations above, which are set on the $d$-dimensional torus $\T^d$ with $d \geq 2$, the unknown $\rho(t,x) \in \R$ is a scalar function which (roughly) represents a density function of particles suspended in a non-Newtonian fluid. The vector quantity $u(t,x) \in \R^d$ is the velocity field of the fluid, and $\pi(t,x) \in \R$ is the pressure field. Let us specify that the initial data of the system is given by $\rho_0(x) \in \R$, and that we consider throughout the rest of the article a null mean assumption on the velocity field, \textsl{i.e.} $\fint u = 0$. Finally, $g \in \R^d$ is the (constant) gravity, $p > 1$ is a fixed parameter and $\nu : \R \longrightarrow \R_+$ is a given continuous function which is aloud to degenerate at zero, namely one can have $\nu(0) = 0$. Here and henceforth, we denote by $Du$ the strain rate tensor, which is the matrix defined by $[Du]_{ij} = \frac{1}{2}(\partial_i u_j + \partial_j u_i)$.

\medskip

Problem \eqref{ieq:ST} appears in the mathematical description of particles suspended in a fluid. As such, it is a non-Newtonian alternative of the usual Stokes-Transport system
\begin{equation}\label{ieq:STNewtonian}
    \begin{cases}
        \partial_t \rho + u \cdot \nabla \rho = 0 \\
        - \Delta u + \nabla \pi = \rho g \\
        \D(u) = 0,
    \end{cases}
\end{equation}
which was derived from microscopic models by Höfer \cite{Hofer2018} and Mecherbet \cite{Mecherbet2019} and corresponds to \eqref{ieq:ST} when taking $p=2$ and $\nu(r) \equiv 1$. It has also been present in the physical literature as a continuum model for sedimentation, see \cite{Feuillebois} and the references in \cite{Hofer2018}. While we will give more detail in the rest of the introduction, already we point out that problem \eqref{ieq:STNewtonian} has been intensively studied by mathematicians in the past five years, as different authors have been exploring questions related, for example, to well-posedness \cite{Leblond}, \cite{MS}, \cite{Cobb}, controllability or regularity of particle trajectories \cite{MS}, or stability of stratified equilibria \cite{DGL}. In addition, it should be noted that problem \eqref{ieq:ST} belongs to a class of PDEs called active scalar equations (\textsl{e.g.} generalized SQG equations), which is currently under deep scrutiny (see the discussion below).

\medskip

As we will see, the study of the non-Newtonian generalization \eqref{ieq:ST}, which we believe to be new in the mathematical literature, presents some additional difficulties in comparison with the usual Stokes-Transport system \eqref{ieq:STNewtonian}, due to the fact that the leading order terms in the second equation are nonlinear. This will be reflected in the fact that the methods we use in this article to construct solutions are entirely different from those used in previous works for \eqref{ieq:STNewtonian}, and are much more involved. In addition, the introduction of a viscosity multiplier $\nu(r)$ that is aloud to degenerate at zero $\nu(0) = 0$ also complicates matters.\\

\medskip

In the rest of this introduction, we will start by shortly describing the physical background of the PDE system \eqref{ieq:ST}, before discussing previous works on the (Newtonian) Stokes-Transport equations \eqref{ieq:STNewtonian} and related PDEs. Finally, we will state and comment our main result, concerning the existence of global weak solutions for \eqref{ieq:ST}.

\subsection{Physics of the Problem}

Let us start by describing the physical problem from which system \eqref{ieq:ST} originates. We consider an incompressible non-Newtonian fluid of constant density $\bar{\rho} > 0$, inside which lie suspended particles distributed according to the density function $\mu = \bar{\rho} + \rho$. On the one hand, the particles are assumed to be without inertia, so that they are simply transported by the fluid velocity. This gives a first relation
\begin{equation*}
    \partial_t \mu + u \cdot \nabla \mu = 0,
\end{equation*}
which is equivalent to the first equation in \eqref{ieq:ST}. On the other hand, the fluid velocity is assumed to evolve much slower than the particle density, so it is governed by a Stokes equation
\begin{equation*}
    \begin{cases}
        - \D \mathbb{S}[\rho, u] + \nabla P = \mu g \\
        \D(u) = 0,
    \end{cases}
\end{equation*}
where $P$ is the pressure inside the fluid. Note that $\mu$ can be replaced by $\rho$ up to replacing the pressure $P$ by the quantity $\pi(t,x) = P(t,x) - g \cdot x$. In the Stokes equation immediately above, $\mathbb{S}[\rho, u]$ is the viscous stress tensor. While $\mathbb{S}[\rho, u]$ is simply a multiple of the strain rate $Du$ for a Newtonian fluid\footnote{We will only consider the case of a simple fluid, where the shear rate $\dot{\gamma}$ and strain rate $Du$ tensors are equal. We refer to the classical textbook \cite{truesdell} for more on this topic.} it is no more true while considering a more general non-Newtonian setting. In this case, $\mathbb{S}[\rho, u]$ assumes a nonlinear expression of $Du$ and $\rho$ generally having the form

\begin{equation*}
    \mathbb{S}[\rho,u] = \tilde{\nu}(\rho,Du)Du,
\end{equation*}

\noindent the nonlinearity then appears in the viscosity coefficient $\tilde{\nu}$. For example, typical viscosity laws follow polynomial type growth with respect to the strain rate:
\begin{equation}\label{ieq:p-growth-viscosity-law}
    \tilde{\nu}(\rho,Du) = \nu(\rho)(\delta + | Du |^2)^{\frac{p-2}{2}},
\end{equation}

\noindent where $\delta \in \{ 0, 1 \}$ and $p > 1$. Our model \eqref{ieq:ST} covers the case where $\mathbb{S}[\rho, u]$ is a multiple of a power of the strain rate:
\begin{equation}\label{ieq:Viscosity}
    \mathbb{S}[\rho, u] = \nu(\rho) |Du|^{p-2} Du,
\end{equation}
where $p > 1$ is given as a parameter of the fluid. Fixing $\nu \equiv 1$ in \eqref{ieq:p-growth-viscosity-law} for simplicity, we observe that $\tilde{\nu}$ decreases as $| Du |$ increases when $1 < p < 2$, and increases as a function of the growth of $| Du |$ when $p > 2$. These two viscosity regimes are referred to as shear-thinning and shear-thickening behavior respectively, and cover many physically interesting cases. Examples of shear-thinning fluids include biological fluids such as blood, mixtures from the food industry such as ketchup or mayonnaise, or mixtures used in the petrochemical industry such as paints or cement. As far as shear-thickening fluids are concerned, consider, for example, a slightly diluted aqueous solution of cornstarch, or many colloidal suspensions. 

\medskip

Our model is formally\footnote{Strictly speaking, system \eqref{ieq:ST} has not yet been derived from microscopic principles, but is a reasonable generalization of the Newtonian Stokes-Transport problem which has been.} associated with a dilute cloud of inertialess particles suspended in a non-Newtonian fluid following a law of the form \eqref{ieq:Viscosity}, which we call an Ostwald-DeWaele or power-type law. Such a viscosity law holds for a large number of fluids, typically blood, which is shear-thinning, or suspensions of silica particles in polyethylene glycol, which exhibit shear-thickening behavior (see \textsl{e.g.} \cite{CaoChenQianShaYu,MichaudSoutrenon,RudyakTretiakov}).\\

Quite often, non-Newtonian fluids are in fact themselves suspensions of particles in a Newtonian fluid, and the presence of the particles confer non-Newtonian properties to the mixture through their microscopic interaction. For example, blood is a suspension of erythrocytes (\textsl{i.e.} red blood cells) in blood plasma (see \textsl{e.g.} \cite{KamenevaRobertsonSequeira,CardiovascularMath}), which is a Newtonian fluid (and also a suspension of various particles such as proteins, etc.). Likewise, polyethylene glycol in the example above is a Newtonian fluid, while suspensions of silica particles therein are not.

\medskip
 
Consequently, a possible interpretation of our model is to consider the suspension of (larger) particles in a non-Newtonian fluid, which may therefore be, for example, a suspension of (smaller) particles in a Newtonian fluid, but whose presence is established in such a way that the interactions of these particles with each other do not influence the global viscosity of the suspension, \textsl{i.e.} the microscopic interactions they induce (Van Der Waals or Bohr forces) are negligible.

\subsection{Overview of Previous Results on Related Systems}

In this paragraph, we conduct a survey of the mathematical literature on the different topics related to our result.

\subsubsection{Stokes-Transport Equations}

As we have said above, the Stokes-Transport system \eqref{ieq:STNewtonian} was first derived from microscopic principles by Höfer \cite{Hofer2018} and Mecherbet \cite{Mecherbet2019}. Later, Höfer and Schubert \cite{HS} derived a Stokes-Transport system when taking into account the Einstein effective viscosity of the suspended particles. This leads to a viscous strain tensor that depends on the particle density through
\begin{equation*}
    \mathbb{S}[\rho, u] = \nu(\rho) Du,
\end{equation*}
where $\nu$ is a scalar function. Note that this corresponds, in the case $p=2$, to the viscosity law \eqref{ieq:Viscosity}. As mentioned in the previous subsection, this can, for example, describe blood plasma as a suspension of dilute particles whose microscopic interactions are small or negligible. However, the Stokes-Transport system with a fully non-Newtonian viscosity (\textsl{i.e.} depending nonlinearly on $Du$) is, to the best of our knowledge, new in the mathematical literature.

\medskip

The well-posedness issue for the Stokes-Transport system \eqref{ieq:STNewtonian} has been studied by numerous authors. Along with the derivation of the equations, the authors of \cite{Hofer2018} and \cite{HS} also show well-posedness in high regularity spaces. Existence and uniqueness in the space $L^1(\R^3) \cap L^\infty(\R^3)$ under a finite moment condition has subsequently been proved by Mecherbet in \cite{Mecherbet2020}. A similar result was obtained by Leblond in \cite{Leblond} when the equations are set in a bounded domain of $\R^2$ or $\R^3$.

\medskip

It should be noted that the case of two dimensions of space creates difficulties, as the elliptic Stokes problem 
\begin{equation}\label{ieq:StokesNewtonian}
    \begin{cases}
        - \Delta u + \nabla \pi = \rho g \\
        \D(u) = 0
    \end{cases}
\end{equation}
is not well-posed on the whole space $\R^2$. This is due to integrability issues, as the inverse map $\rho \mapsto u$, whose behavior is closely related of the inverse Laplacian $(- \Delta)^{-1}$, naturally sends $L^1$ to the space of functions of Bounded Mean Oscillations $\rm BMO$, which are defined up to the addition of a constant. This phenomenon is known as the Stokes paradox. Nevertheless, some results are available in two dimensions. Leblond \cite{Leblond} has proved the existence and uniqueness of solutions in the strip $\R \times ]0, 1[$ for bounded initial data while Grayer \cite{Grayer2023} has done so in the whole space $\R^2$ for bounded and compactly supported initial data $\rho_0$. The first author \cite{Cobb} has also shown well-posedness in $\R^2$ for initial densities $\rho_0$ in the homogeneous Besov spaces $\dot{B}^0_{1, 1} \cap \dot{B}^0_{2, 1}$ (in particular, $\rho_0$ is of zero average).

\medskip 

Because the inverse Stokes map $u = \Psi (\rho)$ solving \eqref{ieq:StokesNewtonian} behaves essentially as the inverse Laplace operator (it can in fact be written as $\Psi (\rho) = (- \Delta)^{-1} \P (\rho g)$, where $\P$ is the Leray projection), it is apparent that the space $L^1 (\R^3) \cap L^\infty (\R^3)$ is not optimal for solving \eqref{ieq:STNewtonian}. Indeed, the transport structure suggests that it is enough for the velocity to be (Log)-Lipschitz to obtain uniqueness of solutions, while the 3D Stokes map $\Psi$ sends $L^\infty$ to $C^{2 - \epsilon} \subsetneq W^{1, \infty}$. Consequently, Mecherbet and Sueur \cite{MS} have shown global existence and uniqueness for \eqref{ieq:STNewtonian} in the critical space $L^1(\R^3) \cap L^3(\R^3)$ by using Wasserstein distance $W_1$. By Eulerian techniques, the first author of the present paper has also proved \cite{Cobb} an analogous critical well-posedness result for a system with fractional viscosity. Finally, the existence of (possibly) non-unique Lagrangian solutions with $L^1(\R^3)$ initial data has been shown by Inversi in \cite{Inversi}.

\medskip

Many articles have also been written concerning the qualitative description of solutions. For example, following a paper of Elgindi \cite{Elgindi17} on the inviscid Incompressible Porous Media (IPM) equations (which has a comparable structure) Dalibard, Guillod and Leblond \cite{DGL} have explored the long time behavior of solutions in a horizontal channel for initial data that are close to a linearly stratified equilibrium, which involves some mixing effects. Patch solutions, of the form $\rho(t) = \mathds{1}_{\Omega(t)}$ have also been studied by Mecherbet \cite{Mecherbet2020} and Grayer \cite{Grayer2023}. Sueur and Mecherbet \cite{MS} study the exact controllability of the system as well as the analyticity of particle trajectories.

\subsubsection{Active Scalar Equations}

 As previously mentioned, the (non)-Newtonian Stokes-Transport system is a particular case of a very general class of PDEs called active scalar equations. These are scalar transport (or transport-diffusion) equations associated to a velocity field that is given as a function of the scalar unknown. In general, these PDEs take the form
\begin{equation*}
    \begin{cases}
        \partial_t \rho + u \cdot \nabla \rho = 0 \\
        u = T (\rho),
    \end{cases}
\end{equation*}
where $T$ is possibly nonlinear and non-local map and in which we recall that we implicitly consider an initial datum $\rho_0$. In our case, $T = \Psi$ is the inverse Stokes map. Other cases include Surface Quasi-Geostrophic equations (SQG) set on $\R^2$ and where $T$ is given by
\begin{equation*}
    u = T (\rho) := - \nabla^\perp (- \Delta)^{- \beta /2} \rho
\end{equation*}
for some $\beta \in ]0, 2]$. This system is mainly used to model the dynamics of the temperature in the atmosphere. In the case $\beta = 2$, the SQG equations reduce to the 2D Euler system in vorticity form. The SQG equation have been, and continue to be, the object of intense attention by mathematicians, and it is impossible for us here to give a decent picture of the literature. We will simply say that the operator $T = -\nabla^\perp (- \Delta)^{- \beta / 2}$ is skew-symmetric, and so creates a commutator structure in the energy estimates, thus enabling solutions to exists even when the velocity field is less regular than the scalar unknown $\beta < 1$. This fact was discovered by Resnick \cite{Resnick} and exploited to construct weak solutions (see also Marchand \cite{Marchand}). Local well-posedness has also been obtained by using the commutator structure by Chae, Constantin, C\'ordoba, Gancedoo and Wu \cite{CCCGW}. We point out that this commutator structure is absolutely crucial to well-posedness when $\beta < 1$, as shown in the work of Friedlander, Gancedo, Sun and Vicol \cite{FGSV} where ill-posedness is shown for a singular IPM equation, for which the operator $T$ is symmetric.

\medskip

Other active scalar equations include the Burgers equation $T(\rho) = \rho$, the Hamilton-Jacobi equation or the inviscid IPM equation, which is of special interest to us because of its proximity to the Stokes-Transport system: the velocity $u = T(\rho)$ is given by Darcy's law
\begin{equation*}
    \begin{cases}
        u + \nabla \pi = \rho g \\
        \D(u) = 0.
    \end{cases}
\end{equation*}
The question of global well-posedness for the IPM equations is challenging in the extreme. Elgindi \cite{Elgindi17} has proven the existence of global regular solutions near a linearly stratified equilibrium (see also Castro, C\'ordoba and Lear \cite{CCL}), while the first author of the present article has proved local well-posedness in optimal Besov spaces \cite{Cobb}.

\subsubsection{Non-Newtonian Fluids}

In a broader context, several results have been established concerning non-Newtonian flows since the 1960s. The mathematical analysis of power-law type flows, mainly in the case of homogeneous fluids, was first initiated by Ladyzhenskaya \cite{Ladyzhenskaya}, and J.L. Lions \cite{JLL}. In this case, we mainly consider quasilinear parabolic systems of the form

\begin{equation}\label{ieq:NonNewtonianNS}
    \begin{cases}
        \partial_tu + (u\cdot\nabla)u - \D \mathbb{S}[u] + \nabla\pi = f\\
        \D(u) = 0,
    \end{cases}
\end{equation}

\noindent with some boundary conditions, and where the viscous stress tensor is associated with a viscosity of the type \eqref{ieq:p-growth-viscosity-law} for a constant $\rho$. It is then possible to show the existence of global weak solutions for a parameter $p > 2d/(d+2)$,  we refer for example to the work of Diening, R\r{u}\u{z}i\u{c}ka, and Wolf \cite{DieningRuzickaWolf} as well as Bul\'{i}\u{c}ek, Gwiazda, M\'{a}lek, and \'{S}wierczewska-Gwiazda \cite{BGMS} for the homogeneous case (constant density), or Frehse and R\r{u}\u{z}i\u{c}ka \cite{FrehseRuzicka} for the non-homogeneous one (non-constant density).

\medskip

At this point, it is important to point out that the role of the parameter $p > 1$ plays a key role in the mathematical analysis of such models. Firstly, when the fluid is sufficiently shear-thickening, \textsl{i.e.} when $p > 2$ is large enough, it is possible to show the existence and uniqueness of global strong solutions, we mainly refer to the work of M\'{a}lek, Ne\u{c}as, and R\r{u}\u{z}i\u{c}ka \cite{MalekNecasRuzicka} (see also \cite{JLL}). In the more general case where the inequality $p > 2$ is not assumed, problem \eqref{ieq:NonNewtonianNS} may exhibit pathological behavior. While it is a priori possible to show the existence of weak solutions for $p > 1$, note that energy solutions (\textsl{i.e.} of the Leray-Hopf type) can be non-unique according to the work of Burczak, Modena, and Sz\'{e}kelyhidi \cite{BurczakModenaSzekelyhidi} for $1 < p < 2d/(d+2)$, which corresponds to the limiting case of the embedding $W^{1,p} \hookrightarrow L^2$. However, in the absence of a convective term, one expects the existence of good energy solutions from the work of Berselli and R\r{u}\u{z}i\u{c}ka \cite{BerselliRuzicka2} at least on well-chosen domains (see also the work of Gwiazda, \'{S}wierczewska-Gwiazda and Wr\'{o}blewska \cite{GSW}). Note that even the elliptic $p$-Laplacian problem is ill-posed (solutions may be non-unique). For example, Colombo and Tione \cite{ColomboTione} show non-uniqueness when the domain does not satisfy a suitable cone condition over the weak derivatives. It also remains possible to show the existence of local strong solutions for \eqref{ieq:NonNewtonianNS} as proved in the work of Berselli, Diening and R\r{u}\u{z}i\u{c}ka \cite{BerselliDieningRuzicka} as well as Amann's \cite{Amann}.

\medskip

Despite these difficulties, other properties appear when $ 1 < p < 2$. For example, it is possible to show that the energy associated with the flow stops in finite time (see the work of Chupin, Cîndea, and the second author of the present article in \cite{CCL2023}). These results are also valid for the parabolic $p$-Laplacian according to the work of DiBenedetto (see \cite{DiBenedetto}). In particular, this makes it possible to establish global exact controllability properties (see the work of Cîndea and the second author of the present article \cite{CindeaLacour}).

\medskip

Let us focus again on the Stokes-Transport system \eqref{ieq:ST}. In contrast with \eqref{ieq:NonNewtonianNS}, problem \eqref{ieq:ST} is a coupled hyperbolic-elliptic system, the elliptic part of which is a Stokes system with a nonlinear viscosity coefficient. By itself, the elliptic problem has already been intensively studied: Berselli and R\r{u}\u{z}i\u{c}ka \cite{BerselliRuzicka1, BerselliRuzicka3} have examined the existence and regularity of weak solutions. One of the difficulties arising from the nonlinear nature of the Stokes system present in \eqref{ieq:ST} is to take coupling into account when establishing weak solutions, which then comes down to showing that the mapping which associates the density with the velocity field, which we shall refer as the inverse Stokes mapping, is continuous in a suitable topology.

\medskip

Finally, let us observe two facts.  By considering a finite number of particles as rigid bodies immersed in such a fluid, it is possible to show the existence of suitable weak solutions for coupled fluid-particle interaction. We refer to the work of Feireisl, Hillairet, and Ne\u{c}asov\'{a} \cite{FeireislHillairetNecasova} established in a sufficiently shear-thickening frame given by $p \geq 4$. It is also worth to mention the article of Starovoitov \cite{Starovoitov}, who showed the absence of collision between rigid bodies immersed in the fluid. However, our model \eqref{ieq:ST}, which is formally a continuum limit of immersed particles suspended in a power-law type fluid, has not yet been studied in the mathematical literature. Secondly, let us point out that, for the sake of simplicity, we study the problem \eqref{ieq:ST} on the $d$-dimensional torus $\T^d$. This has the double advantage of working on a compact domain while retaining the possibility to resort to Fourier analysis (and in particular Littlewood-Paley theory and paradifferential calculus). It should be noted that the analysis of non-Newtonian flows is often carried out within such a framework, as in the work of Abbatiello and Feireisl \cite{AbbatielloFeireisl}, although generalization to bounded domains of $\R^d$ follows from minor modifications. We expect this to also be the case for our problem.

\subsection{Main Result}

In this paragraph, we comment on the main difficulties linked to the question of existence of solutions for \eqref{ieq:ST}. We then state our main result: the existence of global weak solutions. Finally, we give a short overview of the methods used in our proof.

\subsubsection{Principal Difficulties in the Problem}

As we have explained in the survey of the mathematical literature related to the Stokes-Transport problem above, research has essentially been concerned with the case of Newtonian fluids \eqref{ieq:STNewtonian}. It should be pointed out that the non-Newtonian system \eqref{ieq:ST} is expected to behave very differently: this is due to the fact that the non-Newtonian Stokes problem
\begin{equation}\label{ieq:StokesNonNewtonian}
    \begin{cases}
        - \D \big( \nu(\rho) |Du|^{p-2} Du \big) + \nabla \pi = \rho g \\
        \D (u) = 0
    \end{cases}
\end{equation}
is a nonlinear elliptic problem with possibly non-smooth coefficients depending on the density $\rho$. In such a setting, the inverse Stokes map $u = \Psi(\rho)$ for \eqref{ieq:StokesNonNewtonian} cannot be expected to produce Lipschitz velocity fields if the density only belongs to $L^\infty$. This shows that any global well-posedness result would be (if possible) hard to prove, as one cannot rely on the preservation of Lebesgue norms.

\medskip

In fact, even the existence of global weak solutions may seem challenging at first glance, due to the nature of the nonlinearities in the non-Newtonian Stokes equation \eqref{ieq:StokesNonNewtonian}. Indeed, the basic energy estimates available for \eqref{ieq:StokesNonNewtonian} yield bounds of the form (see Proposition \ref{p:APriori})
\begin{equation}\label{ieq:WeightedLP}
    Du \in L^p \big( \nu(\rho) \dx \big) \subset W^{1, \beta}
\end{equation}
provided that the density is in a Lebesgue space $\rho \in L^q$ of exponent $q$ sufficiently large and $\beta$ is chosen appropriately. This means that there is a lack of compactness to deal with the viscosity law \eqref{ieq:Viscosity} which involves a nonlinear expression $|Du|^{p-2} Du$ of the strain rate. In addition, the estimate \eqref{ieq:WeightedLP} is made in a weighted space, and difficulties may rise when the viscosity multiplier $\nu(\rho)$ degenerates (and we have assumed it may). Finally, the presence of the viscosity multiplier, which is a nonlinear function of the density, creates another set of problems as the usual Lebesgue estimates $\rho \in L^\infty(\R_+ ; L^q (\R^d))$ do not provide compactness to handle it.

\begin{rmk}
    It should be noted that the energy space \eqref{ieq:WeightedLP} is a weighted Lebesgue space with a weight depending on the density. This means that the natural functional framework for problem \eqref{ieq:ST} has the particularity to involve the solution itself. We refer to \cite{Danchin2023} for another example where a ``dynamical'' solution space is involved in a fluid problem. Whether this structure can be used to improve our results is an ongoing work of the authors.
\end{rmk}

\subsubsection{Statement of the Main Result}

Despite these three issues, it is possible to prove the existence of global weak solutions in some distributional sense. This is the purpose of Theorem \ref{t:WeakSolutions} below, which is the main result of this paper. 

\begin{thm}\label{t:WeakSolutions}
    We work in dimension $d \geq 2$. Consider $p \in ]1, + \infty[$ and a function $\nu \in C^{0,\overline{\gamma}}(\R \setminus \{ 0 \}) \cap L^\infty(\R)$ such that $\nu(|r|) \geq \nu_* |r|^\gamma$ for all $|r| \leq 1$ and some fixed constants $\nu_*,\gamma > 0$, and setting $\overline{\gamma} = \min \{ 1,\gamma \}$. Consider exponents $q \in  ]1, 2[$ and $\sigma \in [1, + \infty]$ such that one of the following conditions is satisfied: 
    \begin{enumerate}[(i)]
        \item Sub-critical case: either we have the strict inequality
        \begin{equation}\label{eq:ConditionSubCritical}
            \frac{1}{p} \left( 1 + \frac{\gamma}{\sigma} \right) + \frac{1}{q} - \frac{1}{d} < 1 \; ;
        \end{equation}
        \item Critical case: either we have equality
        \begin{equation}\label{eq:ConditionCritical}
            \frac{1}{p} \left( 1 + \frac{\gamma}{\sigma} \right) + \frac{1}{q} - \frac{1}{d} = 1 \;,
        \end{equation}
        as well as the condition $q \geq \frac{2d}{d+2}$. In particular, this last inequality is always true when $d = 2$.
    \end{enumerate}
    Then, for any initial datum $\rho_0 \in L^q(\T^d)$ such that $1/\rho_0 \in L^\sigma(\T^d)$, there exists a weak solution $\rho \in L^\infty \big( \R_+ ; L^q(\T^d) \big)$ associated to a velocity field $u \in L^\infty(\R_+ ; L^{q'}(\T^d))$ such that $\nu(\rho)|Du|^p \in L^\infty(\R_+ ; L^1 (\T^d))$. In addition, if $\rho_0 \in L^r(\T^d)$ for some $r \in [1, + \infty]$, then 
    \begin{equation*}
        \| \rho(t) \|_{L^r} = \| \rho_0 \|_{L^r}
    \end{equation*}
    for almost all times $t \geq 0$.
\end{thm}

Before discussing the proof of Theorem \ref{t:WeakSolutions}, we bring a couple of points to the reader's attention. First of all, as will appear clearly when we perform energy estimates (Proposition \ref{p:APriori}), the condition
\begin{equation*}
    \frac{1}{p} \left( 1 + \frac{\gamma}{\sigma} \right) + \frac{1}{q} - \frac{1}{d} \leq 1
\end{equation*}
(along with the condition $q \geq \frac{2d}{d+2}$ in the critical case) is necessary for the product $\rho u$ to be well-defined, and hence for the PDE system to hold in the sense of distributions. Although we do use DiPerna-Lions theory in our proof, we point out that the weak solution $\rho$ is \textsl{not} only a renormalized solution of the transport equation, but is a full distributional solution. Secondly, when $\nu(r) \equiv 1$ and $p = 2$ (and so $\gamma = 0$), the non-Newtonian Stokes-Transport equation reduces to the usual Stokes-Transport equation \eqref{ieq:STNewtonian}. In that case, Theorem \ref{t:WeakSolutions} exist provided $\rho_0 \in L^q(\T^d)$ with 
\begin{equation*}
    q \geq q_0 := \frac{2d}{d+2}.
\end{equation*}
This is consistent with the existence of weak solutions proved in \cite{Cobb} for $q > q_0$ (although they are proved unique only for $q\geq d$, \cite{MS}, \cite{Cobb}). The fact that the optimal exponent $q_0$ can be reached is due to the use of DiPerna-Lions theory.

\subsubsection{Overview of the Proof}

In this paragraph, we explain how the difficulties mentioned above are solved in the proof of Theorem \ref{t:WeakSolutions}. 

\medskip

As is usual when dealing with evolution problems, the weak solutions of \eqref{ieq:ST} are constructed as limit points of a sequence of approximate solutions $(\rho_n)$, which solve a system of the form
\begin{equation}\label{ieq:AppSystem}
    \begin{cases}
        \partial_t \rho_n + u_n \cdot \nabla \rho_n = 0 \\
        u_n = S_n v_n\\
        v_n = \Psi(\rho_n),
    \end{cases}
\end{equation}
where $S_n$ is an approximation operator and $u = \Psi(\rho)$ is the inverse Stokes map associating the datum $\rho$ to the solution $u$ of the elliptic problem \eqref{ieq:StokesNonNewtonian}. These approximate solutions fulfill the same energy estimates as smooth solutions, and this fact provides uniform bounds with respect to the parameter $n$, namely
\begin{equation*}
    (\rho_n) \subset L^\infty(L^q) \qquad \text{and} \qquad (u_n, v_n) \subset L^\infty (W^{1, \beta}),
\end{equation*}
for some exponent $1 < \beta < p$.

\medskip

Proving the existence of these approximate solutions already requires some work, as system \eqref{ieq:AppSystem} is a nonlinear set of PDEs. Although the presence of the smoothing operator $S_n$ simplifies things very much, the degeneracy of the viscosity multiplier $\nu(r)$ already creates problems: because the function is required to behave at worst as $\nu(|r|) \geq \nu_* |r|^\gamma$ around zero (a Hölder-type condition), the map $\Psi$ may be non-Lipschitz. When writing an ODE approximation of \eqref{ieq:AppSystem}, this means that we should use the Cauchy-Peano theorem instead of Cauchy-Lipschitz (also known as Picard-Lindelöf) theory.

\medskip

Once the existence of approximate solutions is established, there are two main steps. First, in order to show suitable convergence properties for approximate solutions, it is then necessary to enjoy some continuity properties for the inverse Stokes map $\Psi$. Assuming $\rho_n \to \rho$ in a suitable topology, we can then formally write
\begin{equation*}
    v_n = \Psi(\rho_n) \to \Psi(\rho) = v.
\end{equation*}
In other words, the limit velocity $v$ is a solution of the non-Newtonian Stokes equation with righthand side input $\rho$. This allows the approximation scheme to converge. Nevertheless, the \textsl{a priori} estimates imply that it is then natural to consider convergence for the strong topology of $L^q$ in space, which is not trivial \textsl{a priori}.

\medskip

Then comes the next important idea, namely the use of DiPerna-Lions theory \cite{DpL} to obtain strong convergence of the solutions. By adapting the stability theory of \cite{DpL} for transport equations (see pp. 521--523), it is possible to show that 
\begin{equation*}
    \rho_n \tend \rho \qquad \text{in }L^r_{\rm loc}(L^q)
\end{equation*}
for all $1 < r < + \infty$. We refer to \cite{CF1} where a related method has been used to perform a singular perturbation analysis. 

\subsection*{Notation}

We summarize here the notation and conventions that will be used throughout the article.

\begin{itemize}
    \item Unless explicitly stated otherwise, all the integrals bear over the whole torus $\T^d$. Therefore $\int f = \int_{\T^d} f$.

    \item For any exponent $r \in [1, + \infty]$, we note $r'$ the conjugate exponent defined by the relation $\frac{1}{r} + \frac{1}{r'} = 1$.

    \item Consider a sequence $(f_n)_{n \geq 1}$ of elements in a metric space $X$. If the sequence is bounded in $X$, \textsl{i.e.} if there exists a constant $C > 0$ such that $\| f_n \|_X \leq C$ for all $n$, then we note $(f_n)_{n \geq 1} \subset X$.

    \item Unless otherwise mentioned, the brackets $\langle \, \cdot \, , \, \cdot \, \rangle$ should always be understood in the sense of distributions $\mc D' \times \mc D$.

\end{itemize}

\subsection*{Acknowledgements}

The work of the first author has been supported by Deutsche Forschungsgemeinschaft (DFG, German Research Foundation) Project ID 211504053 - SFB 1060.

\section{Definition of Weak Solutions}

In this section, we define the notion of weak solution for the Stokes-Transport equation \eqref{ieq:ST}. Although this is fairly straightforward, there are a few points that require our attention.

\medskip

First of all, as we have explained above, the Stokes-Transport system \eqref{ieq:ST} is an active scalar equation: this means that the only dynamical unknown is the density $\rho$, and all the other unknowns, the velocity field $u$ and the pressure $\pi$ should be expressed in terms of $\rho$. Then, as a consequence, only the properties of the density should matter when defining a notion of weak solution. In other words, a function $\rho$ is a weak solution if and only if there exists functions $u$ and $\pi$ such that the equation is satisfied in the sense of distributions.

\medskip

Taking this into consideration, we formulate the following definition.

\begin{defi}
    We consider a dimension $d \geq 2$ and an exponent $q \in [1, + \infty]$. Let $\rho_0 \in L^q$ be an initial datum. We say that a function $\rho \in L^\infty(L^q)$ is a \textsl{weak solution} of the Stokes-Transport problem \eqref{ieq:ST} associated to the initial datum $\rho_0$ is the following conditions are satisfied:
    \begin{enumerate}[(i)]
        \item There exists a velocity field $u : \R_+ \times \T^d \tend \R^d$ such that $\nu(\rho)|Du|^{p-1} \in L^1_{\rm loc}(\R_+ \times \T^d)$ and which is, for almost every time $t \in \R_+$, a weak solution of the (nonlinear) Stokes equation: in other words $\D (u) = 0$ in $\mc D' (\R_+ \times \T^d)$, the zero average condition $\fint u = 0$ is fulfilled at all times and, for any divergence-free $\phi \in \mc D (\R_+ \times \T^d ; \R^d)$, we have
        \begin{equation*}
            \iint \nu(\rho) |Du|^{p-2} Du : D \phi \dx \dt = \iint \rho g \cdot \phi \dx \dt \; ;
        \end{equation*}
        \item The velocity field is $u \in L^\infty(L^{q'})$ and $\rho$ is a weak solution of the transport equation with initial datum $\rho_0$, that is, for all $\phi \in \mc D (\R_+ \times \T^d ; \R)$, we have
        \begin{equation*}
            \iint \Big( \rho \partial_t \phi + \rho u \cdot \nabla \phi \Big) \dx \dt + \int \rho_0 \phi(0) \dx = 0.
        \end{equation*}
    \end{enumerate}
\end{defi}

\section{\textsl{A Priori} Estimates}

In this paragraph, we perform the basic energy estimates which will be used throughout the proof. 

\begin{prop}\label{p:APriori}
    Consider a smooth initial datum $\rho_0$ associated to a smooth solution $\rho$ and smooth velocity and pressure fields $u$ and $\pi$. Then, under the assumptions of Theorem \ref{t:WeakSolutions}, we have the following inequalities:
    \begin{equation*}
        \| \rho \|_{L^\infty(L^q)} \leq \| \rho_0 \|_{L^q}
    \end{equation*}
    and
    \begin{equation}\label{eq:APrioriInequalityU}
        \| D u \|_{L^\beta} \lesssim \big\| 1/\rho_0 \big\|^{\frac{\gamma}{p-1}}_{L^\sigma} \| \rho_0 \|_{L^q}^{\frac{1}{p-1}},
    \end{equation}
    where $1 < \beta < + \infty$ is defined by the relation $\frac{1}{\beta} = \frac{1}{p} \left( 1 + \frac{\gamma}{\sigma} \right)$. In addition, we also have the inequality
    \begin{equation}\label{eq:Aprioriq'-beta}
        \| u \|_{L^{q'}} \lesssim \| Du \|_{L^\beta}.
    \end{equation}
\end{prop}

\begin{rmk}
    Notice that, when the viscosity coefficient $\nu(\rho)$ is degenerate, that is when $\gamma > 0$ and $\sigma < + \infty$, then we may experience a loss of regularity when compared to the usual $p$-Laplacian estimates: the velocity field is not $W^{1, p}$ but only $W^{1, \beta}$.
\end{rmk}

\begin{rmk}\label{r:Homogeneous}
    Inequality \eqref{eq:APrioriInequalityU} is consistent with the equation's ``scaling '' properties. Since the viscosity coefficient behaves, at worse, as $\nu_* |\rho|^\gamma$ in a neighborhood of $ \rho = 0$, we see that the Stokes equation
    \begin{equation*}
        - \D \big( \nu(\rho) |D u|^{p-2} Du \big) + \nabla \pi = \rho g
    \end{equation*}
    behaves roughly as (in physicist notation)
    \begin{equation*}
        \big[ \text{density} \big]^\gamma \times \big[ \text{velocity} \big]^{p-1} = \big[ \text{density} \big],
    \end{equation*}
    so that the velocity is ``homogeneous'' to a density at the power $\frac{1 - \gamma}{p-1}$. This is absolutely consistent with the righthand-side of \eqref{eq:APrioriInequalityU}, which also scales as a density to the power $\frac{1 - \gamma}{p-1}$.
\end{rmk}

\begin{proof}[Proof (of Proposition \ref{p:APriori})]
    First of all, thanks to the fact that the velocity field is assumed to be divergence-free, the flow map associated to $u$ preserves the Lebesgue measure, and therefore all Lebesgue norms. We deduce that
    \begin{equation}\label{eq:DensityLqBbound}
        \| \rho(t) \|_{L^q} = \| \rho_0 \|_{L^q} \qquad \text{for all } t \in \R_+.
    \end{equation}
    For the same reasons, we also have
    \begin{equation}\label{eq:InverseDensityBound}
        \big\| 1 / \rho(t) \big\|_{L^\sigma} = \big\| 1 / \rho_0 \|_{L^\sigma}.
    \end{equation}
    Now, let us focus on the finding estimates for the velocity field. For this, we perform an energy estimate in the elliptic equation: by taking the scalar product of the Stokes equation by $u$, we obtain
    \begin{equation*}
        \int \nu (\rho) |D u|^{p-2} Du : \nabla u = \int \rho g \cdot u.
    \end{equation*}
    Because the symmetric and skew-symmetric spaces form orthogonal subspaces for the $d \times d$ matrix scalar product, we deduce that
    \begin{equation}\label{eq:EstimationsEnergie}
        \int \nu(\rho) |Du|^p = \int \rho g \cdot u.
    \end{equation}
    There will be two steps in our estimates. First of all, we will show how the norm $\| Du \|_{L^\beta}$ can be controlled by the lefthand term in this inequality. Once we have done that, we will focus on the righthand side term and examine how to close the estimates.

    \medskip

    We then start by finding an upper bound for $\| Du \|_{L^\beta}$. The idea is to introduce the viscosity multiplier $\nu(\rho)$ in order to obtain force the apparition of the lefthand side term of \eqref{eq:EstimationsEnergie}. We have:
    \begin{equation*}
        Du = \nu(\rho)^{1/p} Du \frac{1}{\nu(\rho)^{1/p}}.
    \end{equation*}
    We remark that the $L^p$ norm of the first factor $\nu(\rho)^{1/p} Du$ is exactly (up to the $1/p$-th power) the lefthand side of \eqref{eq:EstimationsEnergie}. On the other hand, because of our assumption on the viscosity multiplier $\nu(r)$, we may bound the last factor in this product by a power of $1/\rho$, namely
    \begin{equation*}
        \frac{1}{\nu(\rho)^{1/p}} \leq \frac{1/\nu_*^{1/p}}{|\rho|^{\gamma /p}}.
    \end{equation*}
    so that the $L^{\sigma p / \gamma}$ power of $\nu(\rho)^{-1/p}$ will exactly be the $L^\sigma$ norm of $1/\rho$ (again, up to the $1/\sigma$-th power). Applying Hölder's inequality, we immediately obtain
    \begin{equation}\label{eq:DuBound1}
        \begin{split}
            \| Du \|_{L^\beta} & \leq \big\| \nu(\rho) Du \big\|_{L^p} \big\| 1/|\rho|^{\gamma/p} \big\|_{L^{\sigma p / \gamma}} \\
            & \lesssim \left( \int \nu(\rho) |Du|^p \right)^{1/p} \big\| 1/\rho \big\|_{L^\sigma}^{\gamma / p},
        \end{split}
    \end{equation}
    where the exponent $\beta$ is defined by the relation
    \begin{equation*}
        \frac{1}{\beta} = \frac{1}{p} + \frac{\gamma}{\sigma p} = \frac{1}{p} \left( 1 + \frac{\gamma}{\sigma} \right).
    \end{equation*}
    Note that this inequality is \textsl{homogeneous} in the sense of Remark \ref{r:Homogeneous}. 
    
    \medskip

    The second step is now to bound the righthand side in the energy balance equation \eqref{eq:EstimationsEnergie}. Because $\rho$ is bounded in $L^q$, we will try to show that $u$ is bounded in $L^{q'}$. With that in mind, we express $u$ as a function of the strain rate tensor: keeping in mind that $\D(u) = 0$, we may write
    \begin{equation}\label{eq:uAsAFunctionOfDu}
        u = -2(- \Delta)^{-1} \D (Du).
    \end{equation}
    In other words, $u$ is the image of $Du$ by a Fourier multiplier of order $-1$. This makes possible expressing the regularity of $u$ in terms of that $Du \in L^\beta$. Two cases must be distinguished, according to whether the sub-critical or the critical assumptions of Theorem \ref{t:WeakSolutions} hold.

    \medskip

    \textbf{First case:} critical case. First of all, note that the condition $q \geq \frac{2d}{d+2}$ is equivalent to the inequality $\frac{1}{q} - \frac{1}{d} \leq \frac{1}{2}$. In particular, we have $\beta \leq 2$ thanks to the relation $\frac{1}{\beta} + \frac{1}{q} - \frac{1}{d} = 1$. Therefore, applying the Besov embeddings of Proposition \ref{p:BesovFineEmbeddings}, we obtain that $L^\beta \subset B^0_{\beta, 2}$. As $q' \geq 2$ by assumption, Propositions \ref{p:BesovEmbeddingScaling} and \ref{p:BesovFineEmbeddings} therefore provide the chain of embeddings
    \begin{equation*}
        B^1_{\beta, 2} \subset B^0_{q', 2} \subset L^{q'},
    \end{equation*}
    since the exponents satisfy $\frac{1}{q'} = \frac{1}{\beta} - \frac{1}{d}$. By use of Proposition \ref{p:FourierMultiplier}, this leads us to the bound $\| u \|_{L^{q'}} \lesssim \| Du \|_{L^{\beta}}$, and so, by plugging this in the energy balance equation \eqref{eq:EstimationsEnergie}, we have
    \begin{equation*}
        \int \nu(\rho) |Du|^p = \int \rho g \cdot u \leq \| \rho \|_{L^q} \| u \|_{L^{q'}} \lesssim \| \rho \|_{L^q} \| Du \|_{L^\beta}.
    \end{equation*}

    \medskip

    \textbf{Second case:} sub-critical case. In that case, the Proposition \ref{p:BesovEmbeddingScaling} provide the inclusion
    \begin{equation*}
        B^1_{\beta, \infty} \subset B^s_{q', \infty} \subset L^{q'}
    \end{equation*}
    with $s = d \left( 1 - \frac{1}{\beta} - \frac{1}{q} + \frac{1}{d} \right) > 0$. In particular, thanks to Proposition \ref{p:FourierMultiplier}, we obtain the same inequality as above:
    \begin{equation*}
        \int \nu(\rho) |Du|^p = \int \rho g \cdot u \leq \| \rho \|_{L^q} \| u \|_{L^{q'}} \lesssim \| \rho \|_{L^q} \| Du \|_{L^\beta}.
    \end{equation*}

    \medskip
    
    In both cases, we may use these estimates in inequality \eqref{eq:DuBound1} in order to finish the proof of the proposition: first we write
    \begin{equation*}
        \begin{split}
            \| Du \|_{L^\beta} &\lesssim \left( \int \nu(\rho) |Du|^p \right)^{1/p} \big\| 1 / \rho \big\|_{L^\sigma}^{\gamma / p} \\
            & \lesssim \| Du \|_{L^\beta}^{1/p} \| \rho \|_{L^q}^{1/p} \big\| 1 / \rho \big\|_{L^\sigma}^{\gamma / p^2},
        \end{split}
    \end{equation*}
    and then we apply Young's inequality $ab \lesssim \epsilon a^p + \frac{1}{\epsilon}b^{p'}$ in order to absorb the $\| Du \|_{L^\beta}$ factor in the lefthand side, and finally get,
    \begin{equation}\label{eq:estimate-Du-beta-rho}
        \begin{split}
            \| Du \|_{L^\beta} & \lesssim \| \rho \|_{L^q}^{p'/p} \big\| 1 / \rho \big\|_{L^\sigma}^{\gamma p' / p^2} \\
            & \lesssim \| \rho \|_{L^q}^{\frac{1}{p-1}} \big\| 1/\rho \big\|_{L^\sigma}^{\frac{\gamma}{p-1}}.
        \end{split}
    \end{equation}
    The upper bound can be evaluated at initial time $t=0$ thanks to \eqref{eq:DensityLqBbound} and \eqref{eq:InverseDensityBound}.

\end{proof}

\section{Well-Posedness for the Stokes System}

Having this in mind, in order to show that we do define an active scalar equation via system \eqref{ieq:ST}, let us begin by proving that the Stokes problem with nonlinear viscosity

\begin{equation}\label{eq:steady-state-case}
    \begin{cases}
        - \D \left( \nu(\rho) |Du|^{p-2} Du \right) + \nabla \pi = \rho g\\
        \D(u) = 0.
    \end{cases}
\end{equation}

\noindent has a unique solution satisfying suitable estimates. Namely, the following Proposition holds.

\begin{prop}\label{p:steady-state-uniqueness}
    Consider $1 < p,q < + \infty$. For any $\rho \in L^q$ there is a unique solution $u \in W^{1,\beta}$ such that $\avint u \dx = 0$ and satisfying
    \begin{equation}\label{eq:estimate-weak-steady-state}
        \int \nu(\rho) |D u|^p \lesssim \| \rho \|_{L^q}^{\frac{1}{p-1}} \big\| 1 / \rho \big\|_{L^\sigma}^{\frac{\gamma}{p-1}}.
    \end{equation}
    We introduce the following notation: we call $\Psi(\rho)$ the unique solution associated to $\rho$, thus defining a nonlinear map $\Psi : L^q \tend W^{1, \beta}$.
\end{prop}

\begin{proof}
\noindent We introduce the real reflexive Banach space 
\begin{equation}\label{eq:definition-Y}
    Y_{\rho} := \left\{ v \in W^{1,1}\, / \; \avint v \dx = 0,\; \D(v) = 0,\; \int \nu(\rho)\lvert Dv \rvert^p\dx < +\infty \right\}
\end{equation}
\noindent endowed with the norm
\begin{equation*}
    \lVert v \Vert_{Y_{\rho}} := \left(\int \nu(\rho)\lvert Du\rvert^p\dx\right)^{\frac{1}{p}}
\end{equation*}
which is the norm involved on $Y_{\rho}$ by the one defined over $L^p(\T^d,\nu(\rho)\dx)$, since $\nu(\rho)\dx$ is a measure from the positivity almost everywhere and the boundedness of $\nu(\rho)$. We point out that the divergence free and mean value integrals assumption arising in the definition of $Y_{\rho}$ make sense since from our a priori estimates we have $Y_{\rho} \subset W^{1, \beta} \subset \mc D'$. Then, let us point out that the functional $A_{\rho}$ defined over $Y_{\rho}$ by
\begin{equation}\label{eq:definition-functional-Arho}
    A_{\rho}(u) := \frac{1}{p}\int \nu(\rho)\lvert Du \rvert^p \dx - \int\rho g\cdot u \dx
\end{equation}
can be rewritten as
\begin{equation}\label{eq:rewrite-Arho}
    A_{\rho}(u) := \frac{1}{p}\lVert u \Vert_{Y_{\rho}}^p - \int\rho g \cdot u \dx.
\end{equation}
In particular, the strict convexity of Lebesgue norms  implies that $A_{\rho}$ is a strictly convex functional, since for every $(u,v) \in Y_{\rho}$, $u \neq v$ and every $\theta \in (0,1)$:
\begin{align*}
    A_{\rho}(\theta u + (1- \theta)v) &= \frac{1}{p}\lVert \theta u + (1-\theta)v \Vert_{Y_{\rho}}^p - \int\rho g \cdot \left(\theta u + (1- \theta)v\right) \dx\\
    &< \theta\left(\frac{1}{p}\lVert u \Vert_{Y_{\rho}}^p - \int\rho g \cdot u \dx\right) + (1 - \theta)\left(\frac{1}{p}\lVert v \Vert_{Y_{\rho}}^p - \int\rho g \cdot v \dx\right)\\
    &=\theta A_{\rho}(u) + (1-\theta)A_{\rho}(v).
\end{align*}

\noindent The second inequality above being true since the function $t \mapsto t^p$ defined over $\mathbb{R}_+$ is strictly convex. The Euler-Lagrange formulation associated to the minimization of $A_{\rho}$ leads, computing the derivative in the direction of a test function, to the fact that each of its minimizers are a weak solution of \eqref{eq:steady-state-case}, and vice versa (thanks to the convexity of $A_\rho$). It remains to prove the existence and uniqueness of such a minimizer. Here we underline that $A_{\rho}$ is coercive over $Y_\rho$, since we get from \eqref{eq:rewrite-Arho}, using Hölder's inequality,

\begin{equation}\label{eq:coercivity-1}
    A_{\rho}(u) \geq \frac{1}{p}\lVert u \Vert_{Y_{\rho}}^p - \lVert \rho \rVert_{L^q}\lVert u \Vert_{L^{q'}}.
\end{equation}

\noindent Then, thanks to \eqref{eq:Aprioriq'-beta}, we get from \eqref{eq:coercivity-1} 

\begin{equation}\label{eq:coercivity-2}
    A_{\rho}(u) \gtrsim \lVert u \Vert_{Y_{\rho}}^p - \lVert \rho \rVert_{L^q}\lVert Du \rVert_{L^{\beta}}
\end{equation}

\noindent Using once again our \textsl{a priori} estimates it follows from \eqref{eq:coercivity-2} that, combining \eqref{eq:InverseDensityBound} and \eqref{eq:DuBound1} 

\begin{equation}\label{eq:coercivity-3}
    A_{\rho}(u) \gtrsim  \lVert u \Vert_{Y_{\rho}}^p -C(\rho_0)\left(\int\nu(\rho)\lvert Du \rvert^p\dx\right)^{\frac{1}{p}}
\end{equation}

\noindent Finally, Young's inequality applied in the right-hand side of \eqref{eq:coercivity-3} leads to the wished result, namely

\begin{equation*}
       A_{\rho}(u) \gtrsim \left({1 - \varepsilon}\right)\lVert u \Vert_{Y_{\rho}}^p -C(\varepsilon,p',\rho_0)
\end{equation*}

\noindent The previous computations mean that $A_{\rho}$ is a $p$-coercive functional over $Y_{\rho}$. Hence, it follows from \cite[Section 8.2.2., Theorem 2, Theorem 3]{evans}  that it admits a unique minimum in $Y_{\rho}$, once again denoted $u$, this last being a weak solution of \eqref{eq:steady-state-case}, which is hence unique. Testing now into the weak formulation against $u$, then using both estimates \eqref{eq:Aprioriq'-beta} and \eqref{eq:estimate-Du-beta-rho} leads to \eqref{eq:estimate-weak-steady-state}.

\end{proof}

\section{Continuity of the Inverse Stokes Map}

In the paragraph just above, we have shown that the Stokes equation can be uniquely solved: with the notations of Theorem \ref{t:WeakSolutions}, for any $\rho \in L^q$, there exists a unique $u$ solution of the Stokes equation \eqref{eq:steady-state-case} which lies in the space $W^{1, \beta}$, where $\frac{1}{\beta} = \frac{1}{p} \left( 1 + \frac{\gamma}{\sigma} \right)$. In other words, we have defined a nonlinear map
\begin{equation}\label{eq:map-uniqueness}
    \begin{array}{r|l}
         & L^q \longrightarrow W^{1,\beta} \\
        \Psi :& \\
        & \rho \longmapsto u.
    \end{array}
\end{equation}

The purpose of this section is to prove a form of continuity property for the map $\Psi$. Of course, it is too much to expect that $\Psi$ is norm continuous $L^q \tend W^{1, \beta}$. However, by relaxing the topologies, we may still obtain a continuity result. The following Proposition is based on the Minty's trick (see \cite[Lemma 2.13.]{Roubicek}).

\begin{prop}\label{p:continuity-psi}
    Let us assume that $p, q, \gamma, \sigma$ are defined as in Theorem~\ref{t:WeakSolutions} and $\beta$ as in Proposition~\ref{p:APriori}. We consider a sequence $(\rho_n)_{n \in \mathbb{N}}$ such that there exists $\rho_n \rightarrow \rho$ in $L^q$. Then, one have for $\Psi$ given by \eqref{eq:map-uniqueness}
    \begin{equation*}
        \Psi(\rho_n) \wtend \Psi(\rho) \qquad \text{in } W^{1, \beta}.
    \end{equation*}
    In other words, the nonlinear map $\Psi$ is $L^q \tend W^{1, \beta}_w$ continuous, where $W^{1, \beta}_w$ is the Sobolev space $W^{1, \beta}$ equipped with its weak topology.
\end{prop}

\begin{proof}
First, let us introduce for  $\phi \in \mathcal{D}_{\D} := \left\{\varphi \in \mathcal{D}\, /\; \mathrm{div}(\varphi) = 0\right\}$ the functional 
\begin{equation*}
    X_n(\phi) := \int \nu(\rho_n)\Big(\lvert Du_n \rvert^{p-2}D(u_n) - \lvert D\phi \rvert^{p-2}D\phi\Big):(Du_n - D\phi)\dx
\end{equation*}
where $u_n = \Psi(\rho_n)$ is the unique solution of \eqref{eq:steady-state-case} associated to $\rho_n$, as given by Proposition \ref{p:steady-state-uniqueness}. Thanks to inequality \eqref{eq:estimate-weak-steady-state} we are sure that the functional $X_n$ is well-defined on the space $\mc D_{\D}$. In addition, the functionals $X_n$ are nonnegative due to the following lemma.

\begin{lemma}\label{l:monotinicityInequality}
 Consider $n \in \mathbb{N}$ and $1 < p < +\infty$. Then for every regular and divergence-free function $\phi \in \mathcal{D}_{\D}$, we have a monotonicity inequality:
 \begin{equation}\label{eq:monotonicity-inequality}
     X_n(\phi) \geq 0.
 \end{equation}
\end{lemma}

\begin{proof}
We begin with the observation that the functional
\begin{equation*}
A_n(\phi) := \frac{1}{p}\int \nu(\rho_n)\lvert D\phi\rvert^p\dx
\end{equation*}
is convex because $p > 1$. Furthermore, $A_n$ acts as a potential with respect to $X_n$. This means that $X_n(\phi)$ is formally given by the relation
\begin{equation}\label{eq:BracketsConvexity}
    X_n(\phi) = \big\langle {\rm d} A_n (u_n) - {\rm d} A_n (\phi), u_n - \phi \big\rangle,
\end{equation}
where ${\rm d} A_n (f)$ is the differential of $A_n$ evaluated at a point $f$. Still formally, the convexity of $A_n$ shows that the above is non-negative $X_n(\phi)$.

\medskip

The argument as it is is not entirely complete: this is because while $\phi \in C^\infty$ is very regular, it may not be the case of $u_n = \Psi(\rho_n)$ which is defined as the unique solution of a nonlinear elliptic problem. However, we may argue by density: the functional $X_n$ is continuous on the space 
\begin{equation*}
    Y_n := \Big\{ f \in W^{1, 1}, \quad Df \in L^p \big( \nu(\rho_n) \dx \big) \Big\},
\end{equation*}
and the viscosity multiplier $\nu(r)$ is bounded, so $C^\infty$ is a dense subspace of $L^p \big( \nu(\rho_n) \dx \big)$. It is therefore enough to show that the brackets in \eqref{eq:BracketsConvexity} are non-negative when assuming that $u_n$ is smooth.
\end{proof}

\noindent Using the weak formulation of \eqref{eq:steady-state-case} into \eqref{eq:monotonicity-inequality}, we get for all $\phi \in \mc D_{\D}$:

\begin{equation}\label{eq:minty-step-1}
    X_n(\phi) = \int\rho_n(u_n-\phi)g\dx - \int \nu(\rho_n)\lvert D\phi \rvert^{p-2}D\phi:(Du_n - D\phi)\dx \geq 0.
\end{equation}

\noindent We can now pass to the limit on $n$ in this expression. First, we have that $\rho_n \to \rho$ a.e. since $\rho_n \tend \rho$ in $L^q$, and $Du_n \rightharpoonup Du$ in $L^{\beta}$, in view of \eqref{eq:DuBound1} and \eqref{eq:estimate-weak-steady-state}. Also we have that $\nu(\rho_n) \tend \nu(\rho)$ in $L^{\beta '}$. Thus we get:

\begin{equation}\label{eq:first-limit}
    \int \nu(\rho_n)\lvert D\phi \rvert^{p-2}D\phi:(Du_n - D\phi)\dx \underset{n \to +\infty}{\tend} \int \nu(\rho)\lvert D\phi \rvert^{p-2}D\phi:(Du - D\phi)\dx.
\end{equation}

\noindent Moreover, we recall that from \eqref{eq:DuBound1} and \eqref{eq:estimate-weak-steady-state} we have the boundedness $(u_n) \subset W^{1,\beta} \subset L^{q'}$ which leads to weak convergence $u_n \rightharpoonup u$ in $L^{q'}$, up to an extraction. Furthermore, the strong convergence $\rho_n \longrightarrow \rho$ in $L^q$ leads to distributional convergence of the product $\rho_n u_n$. If $\psi \in \mc D$, we have
\begin{equation*}
    \begin{split}
        \big| \langle \rho_n u_n - \rho u , \psi \rangle \big| & = \big| \langle \rho_n - \rho, u_n \psi \rangle \big| + \big| \langle u_n - u, \rho \psi \rangle \big| \\
        & = \| \rho_n - \rho \|_{L^q} \| u_n \psi \|_{L^{q'}} + \big| \langle u_n - u, \rho \psi \rangle \big| \\
        & \longrightarrow 0 \qquad \text{as } n \rightarrow + \infty.
    \end{split}
\end{equation*}

\noindent In other words,
\begin{equation}\label{eq:limit-distribution-n}
    \rho_nu_n \underset{n \to +\infty}{\longrightarrow} \rho u \quad \mathrm{in }\; \mc D'.
\end{equation}

\noindent From \eqref{eq:limit-distribution-n}, \eqref{eq:first-limit} and \eqref{eq:minty-step-1}, we can let $n \rightarrow + \infty$ in the functional $X_n(\phi)$. For every test function $\phi \in \mc D_{\D}$ we have

\begin{equation}\label{eq:second-limit}
    X_n(\phi) \underset{n \to +\infty}{\tend} X(\phi) = \int\rho(u-\phi)g\dx - \int \nu(\rho)\lvert D\phi \rvert^{p-2}D\phi:(Du - D\phi)\dx \geq 0.
\end{equation}

\noindent On the other hand, the sequence $\big( \nu(\rho_n) |Du_n|^{p-2}Du \big)_n$ is bounded in the space $L^{p'}$, thanks to the inequality
\begin{equation*}
    \lVert \nu(\rho_n) \lvert Du_n \rvert^{p-1} \rVert_{L^{p'}}^{p'} \leq \lVert \nu \rVert_{L^{\infty}}^{p'-1}\int\nu(\rho_n)\lvert Du_n \rvert^p\dx \leq \lVert \nu \rVert_{L^{\infty}}^{p'-1}\lVert \rho_0^{-1} \rVert_{L^{\s}}^{\s\left(1 - \frac{p'}{r}\right)}\lVert \rho_0 \rVert_{L^{q}}^{p'}.
\end{equation*}

\noindent This means that it possesses a weak limit $\chi$ up to an extraction: there exists a $\chi \in L^{p'}$ such that
\begin{equation}\label{eq:weak-limit-nonlinearity}
    \nu(\rho_n) \lvert Du_n \rvert^{p-2}Du_n \rightharpoonup \chi \quad \mathrm{in} \quad L^{p'}.
\end{equation}

\noindent The main objective is now to show that $\chi = \nu(\rho) \lvert Du \rvert^{p-2}Du$, and that we have the convergence desired in the equation. For this, we will use the convergence properties of the functional $X_n$ we have explored above. Because the functions $(\rho_n, u_n)$ solve the Stokes problem \eqref{eq:steady-state-case} we get that for all test function $\psi \in \mc D_{\D}$:

\begin{equation}\label{eq:weak-formulation-nn}
    \int \nu(\rho_n)\lvert Du_n \rvert^{p-2} Du_n : D\psi \dx = \int \rho_ng\psi\dx.
\end{equation}

\noindent This implies from \eqref{eq:limit-distribution-n} and \eqref{eq:weak-limit-nonlinearity}, that the function $\chi$ fulfills the following relation: for every $\psi \in \mc D'_{\D}$,

\begin{equation}\label{eq:limit-weak-formulation-nn}
    \int \chi : D\psi \dx = \int \rho g\psi\dx.
\end{equation}

\noindent Combining \eqref{eq:limit-weak-formulation-nn} together with \eqref{eq:second-limit} we get:

\begin{equation}\label{eq:minty-estimate-2}
    X(\phi) = \int \Big(\chi : D(u - \phi) - \nu(\rho)\lvert D\phi \rvert^{p-2}D\phi:(Du - D\phi) \Big) \dx \geq 0.
\end{equation}

\noindent To make use of this inequality and prove the proposition, we would now like to test into \eqref{eq:minty-estimate-2} against some well-chosen $\phi$ functions. So, choosing $\phi = u + \lambda \psi$, for $\psi \in \mc D_{\D}$ and $\lambda > 0$ in \eqref{eq:minty-estimate-2}, then divide by $\lambda$ leads to:

\begin{equation}\label{eq:minty-estimate-3}
    -\int\Big(\chi : D\psi  + \nu(\rho)\lvert D(u + \lambda \psi) \rvert^{p-2}D(v + \lambda\psi): D\psi \Big) \dx \geq 0.
\end{equation}

\noindent By passing to the limit as $\lambda \to 0$ into \eqref{eq:minty-estimate-3}, we infer that

\begin{equation*}
    \int \nu(\rho) \lvert Du \rvert^{p-2}Dv:D\psi\dx \geq \int \chi : D\psi\dx.
\end{equation*}

\noindent Following the same line of arguments with $\phi = u - \lambda\psi$ leads to the converse inequality, and then to the equality:

\begin{equation}\label{eq:minty-limit-2}
    \int \nu(\rho) \lvert Du \rvert^{p-2}Du:D\psi\dx = \int \chi : D\psi\dx.
\end{equation}

\noindent Now, thanks to \eqref{eq:weak-formulation-nn}, we obtain by means of \eqref{eq:minty-limit-2} that $u$ satisfies the weak formulation of \eqref{eq:steady-state-case}, namely

\begin{equation*}
    \int \big(\nu(\rho) \lvert Du \rvert^{p-2}Du:D\psi - \rho g\psi\big) \dx = 0.
\end{equation*}

\end{proof}

\begin{rmk}
    We point out that the structure of the proof of Proposition \ref{p:continuity-psi} is quite robust. In particular, if we had added to the Stokes equation an elliptic term, say $- \nu_0 \Delta^k$ for some constant $\nu_0 > 0$, so as to have instead
    \begin{equation*}
        \begin{cases}
            - \nu_0 \Delta^k u - \D \big( \nu(\rho) |Du|^{p-2}Du \big) + \nabla \pi = \rho g \\
            \D(u) = 0,
        \end{cases}
    \end{equation*}
    then the result would still apply: one could use the methods of Proposition \ref{p:steady-state-uniqueness} to obtain, for every $\rho \in L^q$, the existence and uniqueness of a solution $u \in L^p \big( \nu(\rho) \dx \big) \cap H^k$, noted $\Psi_0(\rho)$, and a straightforward adaptation of the proof immediately above would yield continuity $\Psi_0 : L^q \tend W^{1, \beta}_w$ regardless of the value of $\nu_0$.
\end{rmk}

\section{Approximate Solutions}

In this paragraph, we will construct a family of functions regular $(\rho_n)$ that solve an approximate system, namely
\begin{equation}\label{eq:ApproximateSystem}
    \begin{cases}
        \partial_t \rho_n + \D (\rho_n u_n) = 0\\
        u_n = S_n v_n \\
        - \D \big( \nu(\rho_n) |Dv_n|^{p-2} Dv_n \big) + \nabla \pi_n = \rho_n g\\
        \D(v_n) = 0.
    \end{cases}
\end{equation}
In the above, the operator $S_n$ is the Littlewood-Paley approximation operator given from \eqref{eq:littlewood-paley-blocks} and the ensuing discussion. We equip system \eqref{eq:ApproximateSystem} with the following initial values: for every $n$,
\begin{equation}\label{eq:ApproximateInitialValues}
    \rho_n(0) = S_n \rho_0.
\end{equation}
It follows that the approximate initial data are uniformly bounded in $L^q$ with respect to $q$, that is $\| S_n \rho_n \|_{L^q} \lesssim \| \rho_0 \|_{L^q}$. Our reasoning is the following: formally, taking the limit $n \rightarrow + \infty$ lets us recover the non-Newtonian Stokes-Transport system \eqref{ieq:ST}. In addition, the velocity fields $u_n$ are smooth functions with respect to the space variable (they are in fact trigonometric polynomials), and so the solutions $\rho_n$ are also smooth functions. In particular, the \textsl{a priori} estimates of Proposition \ref{p:APriori} hold for the approximate solutions.

\medskip

However, the approximate system \eqref{eq:ApproximateSystem} is a set of highly nonlinear PDEs, so that even the existence of solutions to \eqref{eq:ApproximateSystem} is not obvious! In the next proposition, we make sure that \eqref{eq:ApproximateSystem} does indeed have global smooth solutions.

\begin{prop}\label{p:ApproximateSolutions}
    Consider any $\ell \geq 0$. Then the initial value problem \eqref{eq:ApproximateSystem}-\eqref{eq:ApproximateInitialValues} has a solution $\rho_n \in L^\infty(L^2) \cap W^{1, \infty}_{\rm loc}(H^\ell)$ such that, for almost every time $t \geq 0$, the function $v_n$ is a solution of the Stokes system \eqref{eq:steady-state-case} with righthand side input $\rho_n$. In particular, inequality \eqref{eq:steady-state-case} holds: for almost all times,
    \begin{equation*}
        \int \nu(\rho_n) |D v_n|^p \lesssim \| \rho_n \|_{L^q}^{\frac{1}{p-1}} \big\| 1 / \rho_n \big\|_{L^\sigma}^{\frac{\gamma}{p-1}},
    \end{equation*}
    and so inequality \eqref{eq:DuBound1} provides $v_n \in L^\infty(W^{1, \beta})$.
\end{prop}

\begin{proof}
In order to show existence of solutions to the approximate system \eqref{eq:ApproximateSystem}, we introduce another approximate system, for which solutions will be shown to exists through the Cauchy-Peano theorem. Let us consider, for any fixed $n$ and for all $N \geq 1$, the system
\begin{equation}\label{eq:ApproximateSystem2}
    \begin{cases}
        \partial_t \rho_N + E_N\D (\rho_N u_N) = 0\\
        u_N = S_n v_N \\
        - \frac{1}{N}\Delta^k v_N- \D \big( \nu(\rho_N) |Dv_N|^{p-2} Dv_N \big) + \nabla \pi_N = \rho_N g\\
        \D(v_N) = 0,
    \end{cases}
\end{equation}
where $E_N$ is a Fourier truncation operator defined by
\begin{equation*}
    \forall f \in \mc S, \qquad \what{E_N f}(\xi) = \mathds{1}_{|\xi| \leq N} \what{f}(\xi).
\end{equation*}
and $k \geq 1$ is an exponent which will be fixed later on. In fact, $E_N$ is the $L^2$-orthogonal projection on the space of functions whose Fourier transform is supported in the ball $B(0, N)$. We equip system \eqref{eq:ApproximateSystem2} with the initial datum 
\begin{equation}\label{eq:AppInitialDatum}
    \rho_N(0) = E_N S_n \rho_0.
\end{equation}

\medskip

\textbf{STEP 1: defining an ODE system}. First of all, we show that for any $N \geq 1$, system \eqref{eq:ApproximateSystem2} possesses a global solution $\rho_N$. We will do this by means of the Cauchy-Peano theorem applied in a subspace of $L^2$. We define a map $F_N : L^2 \tend L^2$ by the following steps, which will be justified immediately below:
\begin{enumerate}
    \item For any $\rho_N \in L^2$, we define $v_N$ to be the unique solution of a Stokes system with improved viscosity:
    \begin{equation}\label{eq:StokesImprovedViscosity}
        \begin{cases}
            \frac{-1}{N} \Delta^k v_N - \D \big( \nu(\rho_N) |D v_N|^{p-2} Dv_N \big) + \nabla \pi_N = \rho_N g \\
            \D(v_N) = 0 \, ;
        \end{cases}
    \end{equation}
    \item We then define the trigonometric polynomial $u_N$ by the relation $u_N = S_n v_N$;
    \item Finally, we set $F_N(\rho_N) := \D(\rho_N u_N)$, which is a trigonometric polynomial, and hence in $L^2$.
\end{enumerate}
In the three steps above, only the first one requires additional explanations, as Proposition \ref{p:steady-state-uniqueness} does not give \textsl{per se} solutions to this modified Stokes equation. However, the proof of Proposition \ref{p:steady-state-uniqueness} can be adapted with straightforward modifications, simply by replacing the functional \eqref{eq:definition-functional-Arho} with
\begin{equation*}
    A_\rho(u) = \frac{1}{2N} \int |\nabla^k u|^2 + \frac{1}{p} \int \nu(\rho) |Du|^p - \int \rho g \cdot u
\end{equation*}
and by working in the space $Y_\rho \cap H^k$ instead of \eqref{eq:definition-Y}. Because we have assumed that $q < 2$, the same computations as in the proof of Proposition \ref{p:steady-state-uniqueness} show that the functional $A_\rho$ is strictly convex and coercive. We deduce, in the same way, that for every $\rho_N \in L^2$, system \eqref{eq:StokesImprovedViscosity} possesses a unique weak solution $v_N \in H^k \cap Y_\rho$ that satisfies $\int v_N = 0$ and
\begin{equation*}
    \| v_N \|_{H^k} \lesssim \sqrt{N} \| \rho_N \|_{L^2}.
\end{equation*}

\begin{rmk}
    Introducing a penalized term acting as an improved viscosity by adding $- \frac{1}{N} \Delta^k$ in the Stokes problem is a necessary step: it allows $v_N$ to be estimated in some space by an upper bound that does not involve $\| 1 / \rho_N \|_{L^\sigma}$, thus enabling the Cauchy-Peano theorem to work in the energy space $L^2$. 
\end{rmk}

\textbf{STEP 2: applying the Cauchy-Peano theorem}. If we wish to apply the Cauchy-Peano theorem to the ODE $\partial_t \rho_N = F_N(\rho_N)$, we must check that $F_N : L^2 \longrightarrow L^2$ is continuous. Consider then two functions $\rho_N, \rho_N' \in L^2$ to which we associate $v_N$ and $v_N'$ solutions of the Stokes problem \eqref{eq:StokesImprovedViscosity}, and $(u_N, u_N') = S_n (v_N, v_N')$. Then, we have
\begin{equation}\label{eq:StabilityPeano2}
    \begin{split}
        \big\| F_N(\rho_N) - F_N(\rho_N') \big\|_{H^M} & \lesssim C(N) \big\| E_N \D \big( (\rho_N - \rho_N') u_N \big) \big\|_{H^M} \\
        & \qquad \qquad + C(N) \big\| E_N \D \big( (u_N - u_N') \rho_N' \big) \big\|_{H^M} \\
        & \lesssim C(N) \| \rho_N - \rho_N' \|_{L^2} \| u_N \|_{L^\infty} + C(N) \| u_N - u_N' \|_{L^\infty} \| \rho_N' \|_{L^2}.
    \end{split}
\end{equation}
We turn our attention to the difference $u_N - u_N'$. We subtract the elliptic equation satisfied by $v_N$ by the one satisfied by $v_N'$, multiply by the function $v_N - v_N'$ and integrate by parts so as to obtain
\begin{equation}\label{eq:StabilityPeano1}
    \begin{split}
        \frac{1}{N} \int \big| \nabla^k v_N - \nabla^k v_N' \big|^2 + &  \int \Big( \nu(\rho_N) |D v_N|^{p-2} D v_N - \nu(\rho_N') |Dv_N'|^{p-2} Dv_N' \Big) : D(v_N - v_N')\\
        & \qquad \qquad = \int (\rho_N - \rho_N') g \cdot (v_N - v_N').
    \end{split}
\end{equation}
Let us decompose the large integral above into two parts, in order to benefit from the monotonicity inequalities fulfilled by the Stokes operator: we have
\begin{multline*}
\int \Big( \nu(\rho_N) |D v_N|^{p-2} D v_N - \nu(\rho_N') |Dv_N'|^{p-2} Dv_N' \Big) : D(v_N - v_N') \\
= \int \nu(\rho_N) \Big( |Dv_N|^{p-2} Dv_N - |Dv_N'|^{p-2} Dv_N' \Big) : D(v_N - v_N') \\
+ \int \Big( \nu(\rho_N) - \nu(\rho_N') \Big) |Dv_N'|^{p-2} Dv_N' : D(v_N - v_N')
\end{multline*}
By using the convexity of the $L^p \big( \nu(\rho_N) \dx \big)$ norm exactly as in Lemma \ref{l:monotinicityInequality}, we see that the first integral in the righthand side above is nonnegative. Therefore, we deduce from equation \eqref{eq:StabilityPeano1} that
\begin{equation*}
    \frac{1}{N} \| v_N - v_N' \|_{H^k} \leq \big\| \nu(\rho_N) - \nu(\rho_N') \big\|_{L^\infty} \| Dv_N'\|_{L^\infty}^{p-1} \| Dv_N - Dv_N' \|_{L^1} + \| \rho_N - \rho_N' \|_{L^2} \| v_N - v_N' \|_{L^2}.
\end{equation*}
By fixing the exponent $k$ so that $k > 1 + d/2$, we may benefit from the Sobolev embedding $H^k \subset W^{1, \infty}$. Moreover, we use the fact that the function $\nu(r)$ is, by assumption, $C^{0, \bar{\gamma}}$-Hölder, where $\bar{\gamma} = \min \{ \gamma, 1 \}$. The combination of these two facts yields the inequality
\begin{equation*}
    \frac{1}{N} \| v_N - v_N' \|_{H^k} \lesssim \| \rho_N - \rho_N' \|_{L^\infty}^\gamma \| Dv_N' \|_{H^k}^{p-1} + \| \rho_N - \rho_N' \|_{L^2}.
\end{equation*}
Plugging this inequality in \eqref{eq:StabilityPeano2} and using the fact that the operator $E_N$ is in fact bounded in $H^{-1} \tend L^2$ shows that the map $F_N : L^2 \longrightarrow L^2$ is $\bar{\gamma}$-Hölder, and therefore continuous. The Cauchy-Peano applied in the \textsl{finite dimensional} space
\begin{equation*}
    X_N := \ker ({\rm Id} - E_N) \qquad \text{endowed with the norm } \| \, . \, \|_{L^2}
\end{equation*}
consequently provides the existence of a local solution $\rho_N \in C^1([0, T_N[ ; X_N)$ for the approximate problem.

\medskip

\textbf{STEP 3: lifespan of the approximate solutions}. The next step is to prove that the lifespan $T_N$ of the approximate solution $\rho_N$ can be extended to reach $T_N = + \infty$. For this, we note that by performing a simple energy estimate in the transport equation in \eqref{eq:ApproximateSystem2}, we obtain conservation of the $L^2$ norms: since $\rho_N = E_N \rho_N$, we have
\begin{equation*}
    \frac{1}{2} \frac{\rm d}{\dt} \int |\rho_N|^2 = - \int \rho_N E_N \D (\rho_N u_N) = - \int \rho_N \D( \rho_N u_N) = 0.
\end{equation*}
and so 
\begin{equation}\label{eq:GlobalBoundsL2}
    \| \rho_N(t) \|_{L^2} = \| E_N S_n \rho_0 \|_{L^2} \leq \| S_n \rho_0 \|_{L^2}.
\end{equation}
We deduce that the approximate solutions always have a bounded time derivative: $\| \partial_t \rho_N \|_{L^2} \leq \max_{\| r \|_{L^2} \leq 1} \| F(r) \|_{L^2}$. In particular, this means that the Cauchy-Peano solution $\rho_N : [0, T_N[ \longrightarrow L^2$ necessarily has a limit at time $T_N^-$, as in that case the sequence
\begin{equation*}
    \rho_N \left(T_N - \frac{1}{j} \right) = \rho_N(0) + \int_0^{T_N - \frac{1}{j}} F(\rho_N) \dt
\end{equation*}
is Cauchy as $j \rightarrow + \infty$. By using the Cauchy-Peano theorem again with initial value $\rho_N(T_N^{\,-})$, we construct an (continuous and not necessarily unique) extension of the ODE solution on $[0, \bar{T_N}[$, which we continue to call $\rho_N$. This extension is $C^1$ with respect to time by virtue of the equation: $\partial_t \rho_N = F_N(\rho_N) \in C^0(L^2)$. We infer that the maximal time $T_N^*$ beyond which there exists no extension of the solution must in fact be $T_N^* = + \infty$, and hence deduce the existence (but not uniqueness) of a global ODE solution $\rho_N : \R_+ \longrightarrow L^2$.

\medskip

\textbf{STEP 4: uniform bounds and weak convergence}. We now prepare to take the limit $N \rightarrow + \infty$ in the approximate system \eqref{eq:ApproximateSystem2}. With this in mind, we use Proposition \ref{p:steady-state-uniqueness} and \eqref{eq:GlobalBoundsL2}, and write the following uniform estimates (recall that $q < 2$).
\begin{equation*}
    (\rho_N) \subset L^\infty(L^2) \subset L^\infty(L^q) \qquad \text{and} \qquad (v_N) \subset L^\infty(W^{1, \beta}) \subset L^\infty(L^{q'}).
\end{equation*}
Concerning the velocity fields $u_N = S_n v_N$, the presence of the regularization operator $S_n$ implies that the $u_n$ are trigonometric polynomials of uniform degree (at most $C2^n$), and so $(u_N) \subset L^\infty(H^{\ell})$ for any $\ell \in \Z$. In addition, by using the transport equation to trade space regularity for time regularity, we see that, provided $\ell > d/2$ so that $H^\ell \subset L^\infty$,
\begin{equation}\label{eq:appTradeRegularity}
    \partial_t \rho_N = - \D (\rho_N u_N) \subset L^\infty (H^{-1}),
\end{equation}
so that we get, for every finite $T > 0$, the uniform bound $\rho_N \subset W^{1, \infty}_T(H^{-1}$. Ascoli's theorem lets us deduce strong convergence of the densities to some limit
\begin{equation*}
    \rho_N \tend_{N \rightarrow + \infty} \rho_n \qquad \text{in } C^0_T(H^{-2}).
\end{equation*}
Concerning the sequences $(v_N)$ and $(u_N)$, the uniform bounds above provide weak-$(*)$ convergence
\begin{equation*}
    v_N \wtend^* v_n \qquad \text{in } L^\infty(W^{1, \beta} \cap L^{q'})
\end{equation*}
and, for all $\ell \in \Z$,
\begin{equation*}
    u_N \wtend^* u_n \qquad \text{in } L^\infty(H^\ell).
\end{equation*}
The high regularity the velocity field has transfers to the density, since the initial datum $\rho_N(0) = E_N S_n \rho_0$ is uniformly bounded with respect to $N$ in $H^\ell$. By invoking Theorem \ref{t:TransportRegularity}, we have, provided that $\ell > 1 + d/2$ so that $H^\ell \subset W^{1, \infty}$,
\begin{equation*}
    \| \rho_N \|_{L^\infty_T(H^\ell)} \leq \| E_N S_n \rho_0 \|_{H^\ell} \exp \left( C \int_0^T \| u_N \|_{H^\ell} \dt \right).
\end{equation*}
Because of the inequality $\| E_N S_n \rho_0\|_{H^\ell} \leq \| S_n \rho_0 \|_{H^\ell}$ and the uniform bounds on the sequence $(u_N)$, we deduce from the above the estimate $(\rho_N) \subset L^\infty_T(H^\ell)$ for all sufficiently large $\ell$. In particular, the interpolation inequality 
\begin{equation*}
    \| \rho_N - \rho_n \|_{L^2} \leq \| \rho_N - \rho_n \|_{H^{-2}}^{1/2} \| \rho_N - \rho_n \|_{H^2}^{1/2}
\end{equation*}
shows that the convergence of the densities $(\rho_N)$ is in fact strong:
\begin{equation}\label{eq:AppStrongL2Convergence}
    \rho_N \tend \rho_n \qquad \text {in } C^0_T(L^2).
\end{equation}
Finally, these high regularity bounds $(\rho_N), (u_N) \subset L^\infty_T(H^\ell)$ show that, by using the transport equation as in \eqref{eq:appTradeRegularity} to trade space for time regularity, we have $(\rho_N) \subset W^{1, \infty}_T(H^{\ell - 1})$.

 \medskip
 
It remains to make sure that the functions $(\rho_n, u_n, v_n)$ are solutions of the target problem \eqref{eq:ApproximateSystem}.

\medskip

\textbf{STEP 5: the transport equation}. We first check that $\rho_n$ is a solution of the transport equation with velocity field $u_n$, which is absolutely straightforward given the strong convergence \eqref{eq:AppStrongL2Convergence} of the densities. Indeed, thanks to the uniform bounds $(u_N) \subset L^\infty(H^\ell)$, we have weak convergence of the product
\begin{equation*}
    \rho_N u_N \wtend \rho_n u_n \qquad \text{in } L^\infty(L^2).
\end{equation*}
Likewise, we have strong convergence of the initial density defined in \eqref{eq:AppInitialDatum},
\begin{equation*}
    E_N S_n \rho_0 \longrightarrow S_n \rho_0 \qquad \text{in } H^\ell .
\end{equation*}
This implies that $\rho_n$ is indeed a solution of the transport equation with the appropriate initial datum:
\begin{equation*}
    \begin{cases}
        \partial_t \rho_n + \D (\rho_n u_n) = 0 \\
        \rho_n(0) = S_n \rho_0.
    \end{cases}
\end{equation*}

\medskip

\textbf{STEP 6: the Stokes equation}. In order to take the limit $N \rightarrow + \infty$ in the Stokes equation, we resort to the continuity properties of the inverse Stokes map, as in Proposition \ref{p:continuity-psi}. Let us explain. We note $\Psi_N : L^q \tend W^{1, \beta}$ the map that associates any $\rho_N \in L^q$ to the weak solution $v_N$ of the Stokes problem \eqref{eq:StokesImprovedViscosity}. By the remarks we made in STEP 1 above, this map is well-defined. Now, as we have already noticed, the presence of the elliptic summand $- \frac{1}{N}\Delta^k$ does not change any of the convexity or coercivity properties required to solve the Stokes problem. More than that, a quick glance at the proof of Proposition shows that the minor modifications (\textit{e.g.} the addition of a non negative term to the functionals $X_n$) of the exact same argument yield a similar ``continuity'' property: assume that we have strong convergence
\begin{equation}\label{eq:AppLqConvergence}
    \rho_N \tend_{N \rightarrow + \infty} \rho_n \qquad \text{in } L^q.
\end{equation}
Then we may deduce the weak convergence 
\begin{equation*}
    \Psi_N (\rho_N) \wtend_{N \rightarrow + \infty} \Psi(\rho_n) \qquad \text{in } W^{1, \beta}.
\end{equation*}
Because $q < 2$; the strong $C^0_T(L^2)$ convergence \eqref{eq:AppStrongL2Convergence} of the sequence $(\rho_N)$ implies that the convergence \eqref{eq:AppLqConvergence} must take place at (almost) every time $t \in [0, T[$, so that $v_n$ is indeed a solution of the elliptic problem associated to $\rho_n$ at almost every time:
\begin{equation*}
    \begin{cases}
        - \D \big( \nu(\rho_n) |Dv_n|^{p-2}Dv_n \big) + \nabla \pi_n = \rho_n g \\
        \D(v_n) = 0.
    \end{cases}
\end{equation*}
This ends the proof of Proposition \ref{p:ApproximateSolutions}.

\end{proof}

\section{Strong Convergence of the Densities}

As explained in the previous subsection, Proposition \ref{p:ApproximateSolutions} provides a sequence of approximate solutions $(\rho_n, u_n, v_n)$ which solve problem \eqref{eq:ApproximateSystem}. We now begin to study the limit $n \rightarrow + \infty$ by focusing on the densities $(\rho_n)$, on which we establish strong convergence.

\begin{prop}\label{p:StrongDensityConvergence}
    Let $(\rho_n, u_n, v_n)$ be the sequence of approximate solutions as given in \ref{p:ApproximateSolutions}. We have, up to an extraction, strong and pointwise convergence of the densities, namely
    \begin{equation*}
        \rho_n \tend \rho \qquad \text{in } L^r_{\rm loc}(L^q) \text{ and a.e. on } \R_+ \times \T^d.
    \end{equation*}
    for every $1 < r < + \infty$. In addition, still up to an extraction, we have weak-$(*)$ convergence of the velocities $u_n \wtend^* u$ in $L^\infty(W^{1, \beta})$ and the function $\rho$ is a weak solution of the transport equation:
    \begin{equation}\label{eq:TransportLimit}
        \begin{cases}
            \partial_t \rho + \D(\rho u ) = 0 \\
            \rho(0) = \rho_0.
        \end{cases}
    \end{equation}
\end{prop}

    \begin{proof}
    For the proof of Proposition \ref{p:StrongDensityConvergence}, we use ideas from DiPerna-Lions theory of transport equations, which is expected to work on our problem because our \textsl{a priori} estimates show that the velocity field is expected to have regularity $u \in L^\infty(W^{1, \beta})$, where $\beta$ is as defined in Proposition \ref{p:APriori}. Indeed, assumptions \eqref{eq:ConditionSubCritical} or \eqref{eq:ConditionCritical} show that we must have $1 < \beta < + \infty$ so that Theorem \ref{t:CZ} applied to \eqref{eq:uAsAFunctionOfDu} gives 
    \begin{equation*}
        \| \nabla u \|_{L^\beta} = 2 \| \nabla (- \Delta)^{-1} \D (Du) \|_{L^\beta} \lesssim \| Du \|_{L^\beta}
    \end{equation*}

    \medskip

    First of all, we notice that due to the form of the approximate system \eqref{eq:ApproximateSystem}, all the bounds of Proposition \ref{p:ApproximateSolutions} in fact hold for the approximate solutions. Namely, we have the following uniform estimates:
    \begin{equation}\label{eq:UniformBoundsDensity}
        (\rho_n)_n \subset L^\infty(L^q), \qquad \text{and} \qquad \big( 1/\rho_n \big)_n \subset L^\infty(L^\sigma),
    \end{equation}
    as well as
    \begin{equation}\label{eq:UniformBoundsVelocity}
        (u_n)_n \subset L^\infty(W^{1, \beta}) \subset L^\infty(L^{q'}) \qquad \text{and} \qquad (v_n)_n \subset L^\infty(W^{1, \beta}) \subset L^\infty(L^{q'}).
    \end{equation}
    One of the main points of DiPerna-Lions theory is to deal with low regularity solutions by using renormalization functions: more precisely, we say that a function $\eta \in C^1(\R) \cap L^\infty(\R)$ is \textsl{admissible} if $\eta' > 0$ everywhere. For one such function $\eta$, we may multiply the first equation in the approximate system \eqref{eq:ApproximateSystem} by $\eta'(\rho_n)$ and obtain a new solution of the transport equation:
    \begin{equation}\label{eq:RenormalizedApproximate}
        \begin{cases}
            \partial_t \eta(\rho_n) + \D \big( \eta(\rho_n) u_n \big) = 0\\
            \eta(\rho_n(0)) = \eta(S_n \rho_0).
        \end{cases}
    \end{equation}
    This equation is much easier to handle: indeed the sequence $\eta(\rho_n)$ is bounded in the space $L^\infty(L^\infty)$. Furthermore, thanks to the fact that $\eta' > 0$, no information is lost, as \eqref{eq:RenormalizedApproximate} implies that $\rho_n$ is a solution of the approximate transport equation in \eqref{eq:ApproximateSystem}.

    \medskip

    Now, thanks to the remark that $\big( \eta(\rho_n) \big) \subset L^\infty(L^\infty)$, we deduce weak-$(*)$ convergence in that space up to an extraction: there exists a $g \in L^\infty(L^\infty)$ such that
    \begin{equation*}
        \eta(\rho_n) \wtend^* g \qquad \text{in } L^\infty(L^\infty).
    \end{equation*}
    In addition, the uniform bounds of \ref{eq:UniformBoundsVelocity} also provide, up to an extraction, the weak-$(*)$ convergence
    \begin{equation*}
        u_n \wtend^* u \qquad \text{in } L^\infty(W^{1, \beta} \cap L^{q'}).
    \end{equation*}
    We will prove that $g$ is a solution of the transport equation with velocity field $u$. For this, we need to obtain strong compactness on the sequence $\big( \eta(\rho_n) \big)$ in order to take the limit in the product $\eta(\rho_n) u_n$. Note that we have the uniform bound $\big( \eta(\rho_n ) \big) \subset L^\infty(L^{q'})$, so that by exploiting equation \eqref{eq:RenormalizedApproximate} and the embedding $L^{q'} \subset B^0_{q', \infty}$, we get
    \begin{equation*}
        \begin{split}
            \partial_t \eta(\rho_n) & = - \D \big( \eta(\rho_n ) \big) \\
            & \subset L^\infty(B^{-1}_{q, \infty}) \subset L^\infty(B^{-1-d/q'}_{\infty, \infty}).
        \end{split}
    \end{equation*}
    The last inclusion above is an application of Proposition \ref{p:BesovEmbeddingScaling}. The above therefore shows that the sequence $\big( \eta(\rho_n) \big)$ is uniformly bounded in the space $W^{1, \infty}_T (B^{- 1 - d/q'}_{\infty, \infty})$ for every finite $T > 0$, in addition to being already bounded in the space $L^\infty(L^\infty) \subset L^\infty(B^0_{\infty, \infty})$. Therefore, an interpolation argument shows that, for a small enough $\theta \in ]0, 1[$, we may write an inclusion in a Hölder-Besov space: 
    \begin{equation*}
        L^\infty(B^0_{\infty, \infty}) \cap W^{1, \infty}_T(B^{-1-d/q'}_{\infty, \infty}) \subset C^{0, \theta}_T(B^{-s/2}_{\infty, \infty}),
    \end{equation*}
    where the regularity exponent $-s/2 < 0$ may be taken as close to zero as desired by taking $\theta > 0$ as close to zero as necessary. In particular, the Ascoli theorem provides strong convergence
    \begin{equation}\label{eq:EtaStrongBesovConvergence}
        \eta(\rho_n) \tend g \qquad \text{in } L^\infty_T(B^{-s}_{\infty, \infty}).
    \end{equation}
    In order to use this strong convergence and prove that $g$ is a solution of the transport equation with velocity field $u$, we state and prove the following product lemma.

    \begin{lemma}\label{l:ProductBesov1}
        Assume that $0 < s < 1$. The function product $(f, h) \mapsto fh$ is continuous in the $B^1_{\beta, \infty} \times B^{-s}_{\infty, \infty} \tend B^{-s}_{\beta, \infty}$ topology.
    \end{lemma}

    \begin{proof}[Proof (of the lemma)]
        This is a direct application of the Bony decomposition. We write the product $fh$ as a sum
        \begin{equation*}
            fh = \mc T_f(h) + \mc T_h(f) + \mc R(f, h).
        \end{equation*}
        On the one hand, the first paraproduct may be evaluated in the following way: for $j \geq -1$, 
        \begin{equation*}
            \begin{split}
                \| \Delta_j \mc T_f(h) \|_{L^\beta} & \lesssim \sum_{|j-m| \leq 4} \| S_{m-1}f \|_{L^\beta} \| \Delta_m h \|_{L^\infty} \\
                & \lesssim 2^{js} \| f \|_{L^\beta} \| h \|_{B^{-s}_{\infty, \infty}} \\
                & \lesssim 2^{js} \| f \|_{B^1_{\beta, \infty}} \| h \|_{B^{-s}_{\infty, \infty}},
            \end{split}
        \end{equation*}
    so that we have $\| \mc T_f(h) \|_{B^{-s}_{\beta, \infty}} \lesssim \| f \|_{B^1_{\beta, \infty}} \| h \|_{B^{-s}_{\infty, \infty}}$. Likewise, the second paraproduct is bounded in the same manner: for any $j \geq -1$,
    \begin{equation*}
        \begin{split}
            \| \Delta_j \mc T_h (f) \|_{L^\beta} & \lesssim \sum_{|j-m| \leq 4} \| S_{m-1}h \|_{L^\infty} \| \Delta_m f \|_{L^\beta} \\
            & \lesssim \frac{1}{s} 2^{js} \| h \|_{B^{-s}_{\infty, \infty}} 2^{-j} \| f \|_{B^1_{\beta, \infty}} \\
            & \lesssim 2^{(1-s)j} \| f \|_{B^1_{\beta, \infty}} \| h \|_{B^{-s}_{\infty, \infty}},
        \end{split}
    \end{equation*}
    and this gives the inequality $\| \mc T_h(f) \|_{B^{1-s}_{\beta, \infty}}$. Finally, the remainder term can be bounded by using Proposition~\ref{p:op}. We have
    \begin{equation*}
        \| \mc R(f, h) \|_{B^{1-s}_{\beta, \infty}} \lesssim \| f \|_{B^1_{\beta, \infty}} \| h \|_{B^{-s}_{\infty, \infty}}.
    \end{equation*}
    The combination of all three inequalities proves the lemma.
    \end{proof}

    Let us apply Lemma \ref{l:ProductBesov1} in order to study the convergence of the product $\eta(\rho_n)u_n$ as we let $n \rightarrow + \infty$. For any function $\phi \in \mc D ([0, T[ \times \T^d)$, we have
    \begin{equation*}
        \big| \big\langle \eta(\rho_n)u_n - gu, \phi \big\rangle \big| \lesssim \| \eta(\rho_n) - g \|_{L^\infty_T(B^{-s}_{\beta, \infty})} \| u_n \|_{L^\infty_T(B^1_{\beta, \infty})} \| \phi \|_{L^1_T(B^{s-1}_{\beta', 1})} + \big| \langle u_n - u, g \phi \rangle \big|.
    \end{equation*}
    Firstly, the strong convergence property \eqref{eq:EtaStrongBesovConvergence} shows that the first summand in the previous upper bound tends to zero as $n \rightarrow + \infty$. Secondly, we know that $u_n$ converges to $u$ in the weak-$(*)$ topology of $L^\infty(L^{q'})$, while the function $g \phi$ is an element of $L^1(L^\infty)$, and so belongs to the predual space $L^1(L^q)$ (recall that $1 \leq q < d$). This means that the bracket in the above inequality also converges to zero. Finally, the initial data for the transport equation \eqref{eq:RenormalizedApproximate} converges strongly: for any $r < + \infty$, dominated convergence yields
    \begin{equation*}
        \eta(S_n \rho_0) \tend \eta(\rho_0) \qquad \text{in } L^r
    \end{equation*}
    for all $r < + \infty$. This shows that the limit $g$ is a weak solution of the initial value problem
    \begin{equation}\label{eq:gTransport}
        \begin{cases}
            \partial_t g + \D (gu) = 0\\
            g(0) = \eta(\rho_0).
        \end{cases}
    \end{equation}
    We will show that in fact $g = \eta(\rho)$, and thereby deduce that $\rho$ is a solution of the transport equation.

    \medskip

    By performing the same steps by replacing the function $\eta$ by $\eta^2$, we see that the functions $\eta(\rho_n)^2$ have a limit
    \begin{equation*}
        \eta(\rho_n)^2 \wtend h \qquad \text{in } L^\infty(L^\infty)
    \end{equation*}
    which is a solution of the initial value problem
    \begin{equation}\label{eq:SquareTransport}
        \begin{cases}
            \partial_t h + \D (hu) = 0\\
            h(0) = \eta(\rho_0)^2.
        \end{cases}
    \end{equation}
    Now, let us show that $g^2$ is also a solution of the initial value problem \eqref{eq:SquareTransport}. We wish to multiply the transport equation \eqref{eq:gTransport} by $g$. In order to make sure we can do this, we first go through a regularization procedure: consider a mollification sequence $(\psi_\epsilon)_{\epsilon > 0}$ and set $g_\epsilon = \psi_\epsilon * g$. Then $g_\epsilon$ solves the equation
    \begin{equation*}
        \partial_t g_\epsilon + u \cdot \nabla g_\epsilon = \big[ u \cdot \nabla, \psi_\epsilon * \big] g,
    \end{equation*}
    where, in the above, the brackets $[A, B] = AB - BA$ represent a commutator of operators, where $\psi_\epsilon * : f \mapsto \psi_\epsilon * f$ is the convolution operator and where $u \cdot \nabla$ must be understood in the weak sense $u \cdot \nabla = \D (fu)$. Multiplying the above by $g_\epsilon$ gives
    \begin{equation*}
        \partial_t (g_\epsilon^2) + \D (g_\epsilon^2 u) = g_\epsilon \big[ u \cdot \nabla, \psi_\epsilon * \big] g.
    \end{equation*}
    To check that $g^2$ also is a solution of the transport equation \eqref{eq:SquareTransport}, we only must check that the righthand side tends to zero as $\epsilon \rightarrow 0^+$. This is the case, thanks to Lemma \ref{l:Commutator}, which provides 
    \begin{equation*}
        \big[ u \cdot \nabla, \psi_\epsilon * \big] g \tend 0 \qquad \text{in } L^1_T(L^\beta).
    \end{equation*}
    Consequently, $g^2 \in L^\infty(L^\infty)$ and $h \in L^\infty(L^\infty)$ solve the same initial value problem \eqref{eq:SquareTransport}. In addition, the transport equation with a $L^\infty(W^{1, \beta})$ coefficient is well-posed in that space, by Theorem \ref{t:TransportWP}, so that we must have $h = g^2$ and therefore
    \begin{equation*}
        \eta(\rho_n)^2 \wtend^* g^2 \qquad \text{in } L^\infty(L^\infty).
    \end{equation*}

    \medskip

    Let us prove that the weak convergence above implies strong convergence of the $\eta(\rho_n)$. By fixing any $T > 0$ and using $\mathds{1}_{[0, T]} \in L^1(L^1)$ as a test function in the weak-$(*)$ convergence, we see that
    \begin{equation*}
        \| \eta(\rho_n) \|^2_{L^2_T(L^2)} = \langle \eta(\rho_n)^2, \mathds{1}_{[0, T]} \rangle \; \tend \; \langle g^2, \mathds{1}_{[0, T]} \rangle =  \| g \|^2_{L^2_T(L^2)}.
    \end{equation*}
    In other words, the sequence $\big( \eta(\rho_n) \big)$ converges weakly-$(*)$ in $L^\infty (L^\infty)$, and therefore weakly in the $L^2_T(L^2)$, and additionally the norms converge. Therefore, Proposition~\ref{p:NormConvergence} implies convergence in the norm topology of $L^2_T(L^2)$, and so
    \begin{equation}\label{eq:EtaStrongConvergence}
        \eta(\rho_n) \tend g \qquad \text{in } L^2_{\rm loc}(L^2),
    \end{equation}
    and, up to extracting again, the convergence is also true almost everywhere in $\R_+ \times \T^d$.

    \medskip
    
    We now show that the pointwise convergence of the $\eta(\rho_n)$ implies convergence almost everywhere of the $\rho_n$. For this, we resort to an appropriate choice of admissible functions $\eta$ and the notion of convergence in measure (see Definition \ref{d:ConvergenceMeasure}). We first prove that the sequence $(\rho_n)$ is convergent in measure on every $[0, T] \times \T^d$ by resorting to a Cauchy criterion (see Proposition \ref{p:MeasureCauchy}). Fix an $\epsilon > 0$ and two indices $m > n \geq 1$. Moreover, we choose a sequence of ``renormalization functions'' $(\eta_k)$ such that $\eta \in C_b(\mathbb{R})$ and
    \begin{equation}\label{eq:admissible-functions}
        \left\{\begin{array}{l}
        \eta_k \in C_b^1(\mathbb{R}) \medskip \\
        0 < \eta_k' \leq 1 \text{ in } \mathbb{R} \medskip \\
        \eta_k(r) = r \text{ in } [-k,k].
        \end{array}\right.
    \end{equation}
    Then, by decomposing the product $[0, T] \times \T^d$ according to whether $|\rho_n|, |\rho_m| \leq k$ or not, we have
    \begin{equation*}
        \begin{split}
            \big\{ |\rho_n - \rho_m| \geq \epsilon \big\} & = \big\{ |\rho_n - \rho_m| \geq \epsilon \big\} \cap \big\{ |\rho_n| \leq k \text{ and } |\rho_m| \leq k \big\} \\
            & \qquad \qquad \qquad \sqcup \big\{ |\rho_n - \rho_m| \geq \epsilon \big\} \cap \big\{ |\rho_n| > k \text{ or } |\rho_m| > k \big\}.
        \end{split}
    \end{equation*}
    On the one hand, the Markov inequality provides the bound 
    \begin{equation*}
        \begin{split}
            {\rm meas} \big\{ |\rho_n| > k \text{ or } |\rho_m| > k \big\} & \lesssim \frac{1}{k^q} \int_0^T \int \Big( |\rho_m|^q + |\rho_n|^q \Big) \dx \dt \\
            & \lesssim T \frac{\| \rho_0 \|_{L^q}^q}{k^q}.
        \end{split}
    \end{equation*}
    On the other hand, on the set where both $|\rho_m|$ and |$\rho_n|$ are smaller than $k$, we know that these functions are equal to, respectively, $\eta_k(\rho_m)$ and $\eta_k(\rho_n)$. We deduce that
    \begin{equation*}
        {\rm meas} \big\{ |\rho_n - \rho_m| \geq \epsilon \big\} \lesssim {\rm meas} \big\{ |\eta_k(\rho_n) - \eta_k(\rho_m)| \geq \epsilon \big\} + T \frac{\| \rho_0 \|_{L^q}^q}{k^q}.
    \end{equation*}
    Now, we had deduced from \eqref{eq:EtaStrongConvergence} that, for any given $k$, the sequence $\big( \eta_k(\rho_n) \big)_n$ converges almost everywhere (to a limit $g_k$). Proposition \ref{p:ConvergenceMeasureAndAE} then shows that the sequence converges in measure, and in particular it must fulfill the Cauchy criterion of Proposition \ref{p:MeasureCauchy}, so that the measure in the inequality immediately above converges to zero as $m > n \geq N \rightarrow + \infty$. By taking the limit superior, we infer that
    \begin{equation*}
        \limss_{m>n\geq N \rightarrow + \infty} {\rm meas} \big\{ |\rho_n - \rho_m| \geq \epsilon \big\} \lesssim T \frac{\| \rho_0 \|_{L^q}^q}{k^q}.
    \end{equation*}
    Since this upper bound can be made as small as desired by taking $k$ as large as necessary, we see that the limit superior must be zero, and so the sequence $(\rho_n)$ must converge in measure. Proposition \ref{p:ConvergenceMeasureAndAE} then asserts that the convergence is also true almost everywhere on $[0, T] \times \T^d$ up to taking an extraction. By taking $T$ as having integral values $T = M \rightarrow + \infty$, we see by means of diagonal extraction that the convergence can in fact be assumed hold almost everywhere on $\R_+ \times \T^d$. Let $f$ be the limit of this convergence:
    \begin{equation*}
        \rho_n \tend f \qquad \text{a.e. on } \R_+ \times \T^d.
    \end{equation*}
    It is a consequence of Proposition \ref{p:MazurConsequence} that in fact $f = \rho$, so that the $\rho_n$ converge to $\rho$ almost everywhere on $\R_+ \times \T^d$.

    \medskip

    Finally, let us show that the convergence of $(\rho_n)$ takes place in the norm topology of $L^r(L^q)$. According to Proposition \ref{p:NormConvergence}, it is enough to show convergence of the norms $\| \rho_n \|_{L^r_T(L^q)}$ for every $T > 0$, as the space $L^r_T(L^q)$ is uniformly convex (since $1 < r < + \infty$). On the one hand, we have, since the flows of the smooth $u_n$ (which are trigonometric polynomials) preserve the Lebesgue norms
    \begin{equation}\label{eq:NormConvergence}
        \| \rho_n(t) \|_{L^q} = \| \rho_n(0) \|_{L^q} = \| S_n \rho_0 \|_{L^q} \tend_{n \rightarrow + \infty} \| \rho_0 \|_{L^q}.
    \end{equation}
    If we show that the preservation of the Lebesgue norms transfers to the limit $\| \rho(t) \|_{L^q} = \| \rho_0 \|_{L^q}$, then the preceding equation is enough to apply Proposition \ref{p:NormConvergence}. Instead of working directly with the density $\rho$, we do this by means of the almost everywhere convergence $\rho_n \tend \rho$ that we have shown just above, and apply it to the ``renormalized'' equation \eqref{eq:RenormalizedApproximate}. Consider a function $\eta_k$ as in \eqref{eq:admissible-functions} and note that, as $\eta_k(\rho_n)$ is a solution of the initial value problem \eqref{eq:RenormalizedApproximate} with $\eta = \eta_k$, we must have
    \begin{equation*}
        \big\| \eta_k (\rho_n(t)) \big\|_{L^q} = \| \eta_k(S_n \rho_0) \|_{L^q} \qquad \text{for a.e. } t \in \R_+.
    \end{equation*}
    However, the convergence $S_n \rho_0 \tend \rho_0$ in $L^q$ and the almost everywhere convergence of the $\rho_n$ implies that, by dominated convergence, and up to taking an extraction,
    \begin{equation*}
        \| \eta_k (S_n \rho_0) \|_{L^q} \tend \big\| \eta_k(\rho_0) \big\|_{L^q} \qquad \text{as } n \rightarrow + \infty
    \end{equation*}
    and
    \begin{equation*}
        \big\| \eta_k(\rho_n(t)) \big\|_{L^q} \tend_{n \rightarrow + \infty} \big\| \eta_k(\rho(t)) \big\|_{L^q} \qquad \text{for a.e. } t \in \R_+.
    \end{equation*}
    Uniqueness of the pointwise limit then shows that
    \begin{equation*}
        \big\| \eta_k(\rho(t)) \big\|_{L^q} = \big\| \eta_k(\rho_0) \big\|_{L^q} \qquad \text{for a.e. } t \in \R_+.
    \end{equation*}
    Lastly, by our choice of $\eta_k$, we note that $|\eta_k(r)| \leq |r|$, and $\eta_k(r) \rightarrow r$ pointwise as $k \rightarrow + \infty$, so that dominated convergence yields
    \begin{equation*}
        \big\| \eta_k(\rho) \big\|_{L^r_T(L^q)} \tend \| \rho \|_{L^r(L^q)} \qquad \text{as } k \rightarrow + \infty
    \end{equation*}
    and
    \begin{equation*}
        \big\| \eta_k(\rho_0) \big\|_{L^q} \tend \| \rho_0 \|_{L^q} \qquad \text{as } k \rightarrow + \infty,
    \end{equation*}
    so that we do indeed have preservation of the Lebesgue norms $\| \rho(t) \|_{L^q} = \| \rho_0 \|_{L^q}$ and \eqref{eq:NormConvergence} allows us to deduce convergence of the norms $\| \rho_n \|_{L^r_T(L^q)} \tend \| \rho \|_{L^r_T(L^q)}$. As explained above, this is enough to invoke Proposition \ref{p:NormConvergence} and thereby end the proof.
\end{proof}

    \begin{rmk}
        We have also shown in the proof that the $L^q$ norms of the solution are preserved $\| \rho(t) \|_{L^q} = \| \rho_0 \|_{L^q}$ for almost all times $t \geq 0$. By arguing exactly in the same way, we can show that if the initial datum also an element of another Lebesgue space, $\rho_0 \in L^r$ for some $r \in [1, + \infty]$, then the same holds for the $L^r$ norm: $\| \rho(t) \|_{L^r} = \| \rho_0 \|_{L^r}$ for almost all times $t \geq 0$.
    \end{rmk}

\section{End of the Proof of Theorem \ref{t:WeakSolutions}}

In this paragraph, we combine all the elements above and show the existence of a weak solution of the non-Newtonian Stokes-Transport problem. Proposition \ref{p:ApproximateSolutions} provides the existence of a family $(\rho_n, u_n, v_n)$ of functions which solve the approximate problem \eqref{eq:ApproximateSystem}, and for which the \textsl{a priori} estimates also hold, see \eqref{eq:UniformBoundsDensity} and \eqref{eq:UniformBoundsVelocity}. In addition, Proposition \ref{p:StrongDensityConvergence} shows that the densities strongly converge (up to an extraction)
\begin{equation}\label{eq:ProofStrongConvergence}
    \rho_n \tend \rho \qquad \text{in } L^r_{\rm loc}(L^q),
\end{equation}
where $\rho$ is a solution of transport equation \eqref{eq:TransportLimit} with initial datum $\rho_0$. Concerning the velocities, we have the weak convergences
\begin{equation*}
    \begin{array}{c}
        v_n \wtend^* v \\
        u_n \wtend^* u
    \end{array}
    \qquad \text{in } L^\infty(W^{1, \beta}).
\end{equation*}
All that remains to show is then that $u = v$ and that $v$ is, at almost every time $t \in \R_+$, a weak solution of the Stokes equation, namely
\begin{equation*}
    \begin{cases}
        - \D \big( \nu(\rho) |Dv|^{p-2} Dv \big) + \nabla \pi = \rho g \\
        \D(v) = 0.
    \end{cases}
\end{equation*}

\medskip

Let us rephrase the problem: if $\Psi : L^q \tend W^{1, \beta}$ is the inverse map introduced in Proposition \ref{p:steady-state-uniqueness}. We know that $v_n(t) = \Psi \big(\rho_n(t) \big)$ for almost every time $t \in \R_+$ and we have to show that $u = v =\Psi\big(\rho(t)\big)$. This will be done by resorting to the continuity property of the map $\Psi$ from Proposition \ref{p:continuity-psi}. Consider a function $\phi \in L^1(L^q)$. Then,
\begin{equation*}
    \begin{split}
        \big\langle v_n, \phi \big\rangle_{L^1(L^{q'}) \times L^\infty(L^q)} & = \int_0^{+ \infty} \int \Psi(\rho_n) \phi \dx \dt \\
        & = \int_0^{+\infty} \langle \Psi(\rho_n), \phi \rangle_{L^{q'} \times L^q} \dt.
    \end{split}
\end{equation*}
However, Proposition \ref{p:continuity-psi} shows that $\Psi$ is continuous with respect to the $L^q \tend W^{1, \beta}_w \subset L^{q'}_w$ topology. In addition, \eqref{eq:ProofStrongConvergence} insures that strong convergence
\begin{equation*}
    \rho_n(t) \tend \rho(t) \qquad \text{in } L^q
\end{equation*}
occurs for almost every time $t \in \R_+$, hence convergence for the brackets
\begin{equation*}
    \big\langle \Psi\big( \rho_n(t) \big) , \phi(t) \big\rangle_{L^{q'} \times L^q} \tend \big\langle \Psi\big( \rho(t) \big) , \phi(t) \big\rangle_{L^{q'} \times L^q}
\end{equation*}
also occurs at almost every time $t \in \R_+$. By dominated convergence, we deduce convergence in the whole space-time bracket, and hence $\Psi(\rho_n) \longrightarrow \Psi(\rho)$. In other words, $v$ is indeed a solution of the Stokes problem for almost every time.

\medskip

The last thing left is to check that $u = v$. Recall that we have $u_n = S_n v_n$, so that the Fourier transforms $\what{u_n}$ and $\what{v_n}$ must coincide on any ball $B(0, R)$, provided that $n$ is taken large enough $n \geq C \log(R)$. Consequently, the limit Fourier transforms $\what{u}$ and $\what{v}$ must also agree on every ball, and so $u = v$. We have finished proving Theorem \ref{t:WeakSolutions}.

\section{Appendix: Different Tools from PDEs and Analysis}

In this appendix, we briefly recall the results derived from DiPerna-Lions theory for transport equations, and from Littlewood-Paley theory and paradifferential calculus in a second step. We concentrate here on the results and definitions used in the present work and refer the interested reader to the references mentioned for further details and developments.

\subsection{Transport Equations}

In this section, we recall some results concerning transport equations which will be used in the article. We are concerned with the linear transport equation 
\begin{equation}\label{eq:Transport}
    \begin{cases}
        \partial_t f + \D (fu) = 0\\
        f(0) = f_0
    \end{cases}
\end{equation}
where $u$ is a divergence-free vector field $\D(u) = 0$, and will be assumed to have regularity $u \in L^1_T(W^{1, \beta})$ with respect to space, for some $\beta \geq 1$. The first result we mention is a well-posedness theorem for \eqref{eq:Transport}.

\begin{thm}[See Proposition II.1 and Theorem II.2 in \cite{DpL}]\label{t:TransportWP}
    Let $q \geq \beta'$ and consider an initial datum $f_0 \in L^q$. Then there exists a unique weak solution $f \in L^\infty(L^q)$ of the transport equation.
\end{thm}

The proof of this Theorem relies on a regularization procedure: let $(\psi_\epsilon)_{\epsilon > 0}$ be a mollification sequence on $\T^d$. Then, by taking the convolution of the transport equation, we obtain the system
\begin{equation*}
    \partial_t f_\epsilon + \D (f_\epsilon u) = \big[ u \cdot \nabla, \psi_\epsilon * \big] f,
\end{equation*}
where $f_\epsilon = \psi_\epsilon * f$ and the commutator is to be understood in the weak sense: we have 
\begin{equation*}\big[ u \cdot \nabla, \psi_\epsilon * \big] f = \D(f_\epsilon u) - \psi_\epsilon * \D(fu).
\end{equation*}
The following classical lemma shows that the commutator in fact converges to zero as the mollification parameter does.

\begin{lemma}[Lemma II.1 in \cite{DpL}]\label{l:Commutator}
    Assume that $u \in L^1(W^{1, \beta})$ and fix a function $g \in L^\infty(L^q)$ for some $q \geq \beta'$We then have
    \begin{equation*}
        \big[ u \cdot \nabla, \psi_\epsilon * \big] g \tend 0 \qquad \text{in } L^1(L^\alpha),
    \end{equation*}
    where $\alpha$ is given by $\frac{1}{\alpha} = \frac{1}{\beta} + \frac{1}{q}$.
\end{lemma}

Finally, we cite one last result, which has to do with propagation of regularity: the solution $f$ can be as regular as the velocity field and the initial datum allows.

\begin{thm}[See for example Theorem 3.19 in \cite{BCD}]\label{t:TransportRegularity}
    Consider $s > 1 + d/2$ and assume that $u \in L^1(H^s)$ and $f_0 \in H^s$. Then the unique solution $f$ of the transport equation \eqref{eq:Transport} given by theorem \ref{t:TransportWP} has regularity $f \in L^\infty_{\rm loc} (H^s)$ and we have, for every $T > 0$,
    \begin{equation*}
        \| f \|_{L^\infty_T} \leq \| f_0 \|_{H^s} \exp \left( C \int_0^T \| \nabla u \|_{H^{s-1}} \dt \right).
    \end{equation*}
\end{thm}

\subsection{Besov Spaces}

In this subsection, we recall some basics concerning non-homogeneous Besov spaces on $\T^d$ and some of their properties that are useful to us in this article. First, we recall that the Littlewood-Paley decomposition is based on a dyadic partition of the unity in frequency space, \textsl{i.e.} we can find a radially symmetric function with compact support $\chi \in \mc D \big( ]-\pi, \pi[^d \big)$ such that the mapping $r \mapsto \chi(re)$ is decreasing for all $e \in \R^d$, and
\begin{equation*}
\chi(x) = 1 \text{ for } |x| \leq 1 \qquad \text{and} \qquad \chi(x) = 0 \text{ for } |x| \geq 2,
\end{equation*}
then, writing $\varphi(\xi) = \chi(\xi) - \chi(2 \xi)$ and $\varphi_j(\xi) = \varphi(2^{-j}\xi)$, we obtain for every $\xi \in \Z^d$ (see \cite[Proposition 2.10.]{BCD})
\begin{equation*}
 1 = \chi(\xi) + \sum_{j \geq 0} \varphi_j(\xi).
\end{equation*}
which is a partition of unity in frequency. This leads to the definition of Littlewood-Paley blocks given by

\begin{equation}\label{eq:littlewood-paley-blocks}
\begin{array}{ll}
\Delta_j = 0& \text{ if } j \leq -2,\\
\Delta_{-1} = \chi(D),&\\
\Delta_j = \varphi_j(D) & \text{ for } j \geq 0.
\end{array}
\end{equation}

Next, we can define a low-frequency truncation operator, given by $S_j = \chi(2^{j-1} D)$. We point out that the operators $\Delta_j$ and $S_j$ are respectively scaled versions of $\varphi(D)$ and $\Delta_{-1}$, and thus for all $q \in [1, + \infty]$ these operators are uniformly bounded in the $L^q \tend L^q$ topology. We then can write for instance
\begin{equation*}
{\rm Id} = \sum_{j \geq -1} \Delta_j;
\end{equation*}
this equality being known as the non homogeneous Littlewood-Paley decomposition, and holds over $\mathcal{D}'(\T^d)$ (see \cite[Proposition 2.12., Proposition 2.13.]{BCD}).

\medskip

One of the main properties linked with Littlewood-Paley decomposition is given by the fact that it is possible to estimate the derivatives of a distribution in terms of its frequencies:  given a distribution $u$, the Fourier transform of $\Delta_ju$ has its support included in an annulus of size more or less $2^j$, thus the derivative of $u$ will act as a multiplication by $2^j$. This is the object of the following result.

\begin{prop}[Bernstein inequalities, {\cite[Lemma 2.1.]{BCD}}]
Let  $0<r<R$. There exists a constant $C > 0$ such that for any nonnegative integer $k$, any couple $(p,q)$ in $[1,+\infty]^2$, with  $p\leq q$,  and any function $u\in L^p$,  we  have, for all $\lambda>0$,
$$
\displaylines{
{\rm supp}\, ( \widehat u ) \subset   B(0,\lambda R)\quad
\Longrightarrow\quad
\|\nabla^k u\|_{L^q}\, \leq\,
 C^{k+1}\,\lambda^{k+d\left(\frac{1}{p}-\frac{1}{q}\right)}\,\|u\|_{L^p}\;;\cr
{\rm supp}\, ( \widehat u ) \subset \{\xi\in\Z^d\,|\, r\lambda\leq|\xi|\leq R\lambda\}
\quad\Longrightarrow\quad C^{-k-1}\,\lambda^k\|u\|_{L^p}\,
\leq\,
\|\nabla^k u\|_{L^p}\,
\leq\,
C^{k+1} \, \lambda^k\|u\|_{L^p}\,.
}$$
\end{prop}   

We are now able to define the wished non homogeneous Besov spaces, which are Banach spaces.

\begin{defi}[Non-homogeneous Besov space]
Let $s\in\R$ and $1\leq p,r\leq+\infty$. The non-homogeneous Besov\index{Space!Besov} space $B^{s}_{p,r}\,=\,B^s_{p,r}(\T^d)$ is defined as the set of tempered distributions $f \in \mc S'$ for which
$$
\|f\|_{B^{s}_{p,r}}\,:=\,
\left\|\left(2^{js}\,\|\Delta_j f \|_{L^p}\right)_{j \geq -1}\right\|_{\ell^r}\,<\,+\infty\,.
$$
\end{defi}

Similar to the case of Sobolev spaces $W^{s,p}$, the parameter $s \in \R$ acts as a regularity index and the parameter $p$ as an integrability exponent. As a fact, Bernstein inequalities directly leads to the embeddings
\begin{equation}\label{eq:crudeEmbed}
B^k_{p, 1} \subset W^{k, p} \subset B^k_{p, \infty},
\end{equation}
these embbeddings holding for all $k \in \N$ and for all $p \in [1, + \infty]$. Furthermore, we emphasize that another consequence of Bernstein inequalities is that it also leads to some embeddings between Besov spaces. More precisely, we have:

\begin{prop}[See Proposition 2.71 in \cite{BCD}]\label{p:BesovEmbeddingScaling}
    Consider $s_1 \in \R$ and $q_1, q_2, r_1, r_2 \in [1, + \infty]$ such that $q_1 \leq q_2$ and $r_1 \leq r_2$, then the inclusion $B^s_{q_1, r_1} \subset B^{s_2}_{q_2, r_2}$ holds with
    \begin{equation*}
        s_2 := s_1 - d \left( \frac{1}{q_1} - \frac{1}{q_2} \right).
    \end{equation*}
\end{prop}

There are embeddings that are finer than those presented in \eqref{eq:crudeEmbed}. More precisely, it is a matter of devoting a special attention to the role of the index $r$ of the space $B^s_{p, r}$.

\begin{prop}[Theorems 2.40 and 2.41, pp. 79--82, in \cite{BCD}]\label{p:BesovFineEmbeddings}
    Consider $q \in [1, 2]$ and $r \in [2, + \infty[$. Then we have the following continuous embeddings:
    \begin{equation*}
        L^q \subset B^0_{q, 2} \qquad \text{and} \qquad B^0_{r, 2} \subset L^r.
    \end{equation*}
    and also
    \begin{equation*}
        B^0_{q, q} \subset L^q \qquad \text{and} \qquad L^r \subset B^0_{r, r}.
    \end{equation*}
\end{prop}

The following result allows us to establish some continuity properties on Fourier multipliers dealing with dyadic blocks.

\begin{prop}[See Lemma 2.2 in \cite{BCD}]\label{p:FourierMultiplier}
    Consider a Fourier multiplier $m(\xi)$ whose symbol is a smooth function away from the origin $m \in C^\infty (\R^d \setminus \{ 0 \})$ such that there is an order $M \in \R$ with the following property: for all $\alpha \in \N^d$ with $|\alpha| \leq d+2$, there exists a constant $C_\alpha > 0$ with
    \begin{equation*}
        \forall \xi \neq 0, \qquad |\partial^\alpha m (\xi)| \leq C_\alpha |\xi|^{M - |\alpha|}.
    \end{equation*}
    Then, for any Lebesgue exponent $q \in [1, + \infty]$ and any $f \in \mc D'(\T^d)$, we have
    \begin{equation*}
        \forall j \geq -1, \qquad \| m(D) \Delta_j f \|_{L^q} \leq C(d) 2^{Mj} \| \Delta_j f \|_{L^q}.
    \end{equation*}
\end{prop}

By dealing only with dyadic blocks, one cannot completely understand low-frequency behavior. Thus, other tools are required, such as the use of Bernstein inequalities for example. Another possibility is to consider the Calder\'{o}n-Zygmund theory, whose essential result is presented thereafter.

\begin{thm}[Calder\'on-Zygmund, see { \cite[Theorem 4.2.2.]{Grafakos}}]\label{t:CZ}
    Consider a Fourier multiplication operator $m(D)$ whose symbol is a homogeneous function of degree zero $m \in C^\infty(\R^d \setminus \{ 0 \})$. Then for any $1 < q < + \infty$, the operator $m(D) : \mc D(\T^d) \tend \mc D'(\T^d)$ has a unique bounded extension
    \begin{equation*}
        m(D) : L^q \tend L^q
    \end{equation*}
\end{thm}

We now introduce the paraproduct operator. Roughly speaking, this involves showing that the product of two elements $u$ and $v$ of $\mc S'$ can be decomposed into Cauchy series, \textsl{i.e.} we have

\begin{equation}\label{eq:bony}
u\,v\;=\; \mc T_u(v)\,+\, \mc T_v(u)\,+\, \mc R(u,v)\,,
\end{equation}

where we set
\begin{equation*}
\mc T_u(v)\,:=\,\sum_jS_{j-1}u\Delta_j v,\qquad\mbox{ and }\qquad
\mc R(u,v)\,:=\,\sum_j\sum_{|j'-j|\leq1}\Delta_j u\,\Delta_{j'}v\,.
\end{equation*}

The term $\mc R$ is called the remainder operator, and $\mc T$ is called the paraproduct, while the equation \eqref{eq:bony} is referred as the Bony decomposition. We then have the following result.

\begin{prop}[See {\cite[Theorem 2.85.]{BCD}}]\label{p:op}
For any $(s,p,r)\in\R\times[1,+\infty]^2$ and $t>0$, the paraproduct operator 
$\mc T$ maps continuously $L^\infty\times B^s_{p,r}$ in $B^s_{p,r}$ and  $B^{-t}_{\infty,\infty}\times B^s_{p,r}$ in $B^{s-t}_{p,r}$.
Moreover, the following estimates hold:
$$
\| \mc T_u(v)\|_{B^s_{p,r}}\,\lesssim \,\|u\|_{L^\infty}\,\|\nabla v\|_{B^{s-1}_{p,r}}
$$
as well as
\begin{equation*}
\| \mc T_u(v)\|_{B^{s-t}_{p,r}}\,\lesssim \frac{1}{t} \|u\|_{B^{-t}_{\infty,\infty}}\,\|\nabla v\|_{B^{s-1}_{p,r}} \qquad \text{and} \qquad \| \mc T_u(v)\|_{B^{s-t}_{p,r}}\,\lesssim \frac{1}{t} \|u\|_{B^{-t}_{p,\infty}}\,\|\nabla v\|_{B^{s-1}_{\infty,r}}.
\end{equation*}

For any $(s_1,p_1,r_1)$ and $(s_2,p_2,r_2)$ in $\R\times[1,+\infty]^2$ such that 
$s_1+s_2>0$, $\frac{1}{p}:= \frac{1}{p_1}+ \frac{1}{p_2}\leq 1$ and~$\frac{1}{r}:= \frac{1}{r_1} + \frac{1}{r_2} \leq 1$, the remainder operator $\mc R$ maps continuously~$B^{s_1}_{p_1,r_1}\times B^{s_2}_{p_2,r_2}$ into~$B^{s_1+s_2}_{p,r}$. Also, we have the estimates over the remainder:

\begin{equation*}
\lVert  \mc R(u,v) \rVert_{B^{s_1 + s_2}_{p,r}} \lesssim \lVert u \rVert_{B^{s_1}_{p_1,r_1}}\lVert v \rVert_{B^{s_2}_{p_2,r_2}}
\end{equation*}

In the case $s_1+s_2=0$, provided $r=1$, operator $\mathcal{R}$ is continuous from $B^{s_1}_{p_1,r_1}\times B^{s_2}_{p_2,r_2}$ with values
in $B^{0}_{p,\infty}$.
\end{prop}

\subsection{Convergence Lemmas}

In this section, we have gathered a number of eclectic results from functional analysis and measure theory which we will use throughout the article. We start by presenting a notion that is neighbor to almost everywhere convergence: convergence in measure.

\begin{defi}[Convergence in measure]\label{d:ConvergenceMeasure}
    Consider $T > 0$ and a family of measurable functions $f, f_n : [0, T] \times \T^d \tend \R$, for $n \geq 1$. We say that $(f_n)$ converges in measure to $f$ if and only if for all $\epsilon > 0$,
    \begin{equation*}
        {\rm meas} \big\{ |f_n - f| \geq \epsilon \big\} \tend 0 \qquad \text{as } n \rightarrow + \infty.
    \end{equation*}
    Remark that convergence in measure is associated to a metric $d$, defined by
    \begin{equation*}
        d(f,g) := \sum_{n = 1} \frac{1}{2^n} \min \Big( 1, {\rm meas} \big\{ |f-g| \geq 2^{-n} \big\} \Big).
    \end{equation*}
\end{defi}

The next proposition states that the topology of the convergence in measure is complete.

\begin{prop}[Cauchy criterion, see {\cite[Theorem 2.30.]{Folland}}]\label{p:MeasureCauchy}
    Consider $T > 0$ and a family of measurable functions $f_n : [0, T] \times \T^d \tend \R$. Then the sequence $(f_n)$ converges in measure to a measurable function $f$ if and only if the following Cauchy criterion is satisfied: for all $\epsilon, \epsilon' > 0$, there exists a rank $N \geq 1$ such that, for all $m > n \geq N$ we have
    \begin{equation*} 
        {\rm meas} \big\{ |f_n - f_m| \geq \epsilon \big\} \leq \epsilon'.
    \end{equation*}
    In other words, the sequence $(f_n)$ converges in measure if and only if, for all $\epsilon > 0$, we have
    \begin{equation*}
        {\rm meas} \big\{ |f_n - f_m| \geq \epsilon \big\} \tend 0 \qquad \text{as } m > n \geq N \rightarrow + \infty.
    \end{equation*}
\end{prop}

Finally, we gather a few results concerning the links existing between convergence in measure and convergence almost everywhere.

\begin{prop}[See {\cite[Chapter V.13.]{Doob}}]\label{p:ConvergenceMeasureAndAE}
    Consider $T > 0$ and a family of measurable functions $f_n : [0, T] \times \T^d \tend \R$, for $n \geq 1$. Then the following statements hold:
    \begin{enumerate}[(i)]
        \item If the sequence $(f_n)$ converges almost everywhere on $[0, T] \times \T^d$, then the sequence converges in measure.
        
        \item Conversely, if the sequence $(f_n)$ converges in measure, then there is an extracted sequence which converges almost everywhere on $[0, T] \times \T^d$.
    \end{enumerate}
\end{prop}

A second part of this subsection is concerned with weak convergence, and its relation to almost everywhere or strong convergence. 

\begin{prop}\label{p:MazurConsequence}
    Consider $1 < q, r < + \infty$, and a sequence of functions $f_n : \R_+ \times \T^d \tend \R$ such that there is convergence
    \begin{equation*}
        \begin{split}
            &f_n \wtend_{n \rightarrow + \infty} f \qquad \text{in } L^r(L^q) \\
            &f_n \tend_{n \rightarrow + \infty} g \qquad \text{a.e.}
        \end{split}
    \end{equation*}
    The both limits must be equal $f = g$.
\end{prop}

\begin{proof}
    The proof is an application of Mazur's lemma (Corollary 3.8 in \cite{Brezis}). The weak convergence (in the reflexive space $L^r(L^q)$) provides the existence of a sequence of convex combinations of the $f_n$ which converge strongly in that space. More precisely, there exists a set of coefficients $\lambda_n(k) > 0$, with $n \leq k \leq A_n$, $a_n \geq n$, and $n \geq 1$, such that $\sum_k \lambda_n(k) = 1$ and
    \begin{equation*}
        R_n := \sum_{k = n}^{A_n} f_k \tend f \qquad \text{in } L^r(L^q).
    \end{equation*}
    Now consider a point $(t, x) \in \R_+ \times \T^d$ such that the convergence $f_n(t,x) \tend g(t,x)$ holds. Then, by comparing $R_n$ and $g$ at that point, we obtain
    \begin{equation*}
        \begin{split}
            \big| R_n(t,x) - g(t,x) \big| & \leq \sum_{k=n}^{1_n} \lambda_n(k) \big| \rho_k(t,x) - g(t,x) \big| \\
            & \leq \sup_{k \geq n} \big| \rho_k(t,x) - g(t,x) \big| \; \tend_{n \rightarrow} \; 0.
        \end{split}
    \end{equation*}
    We deduce that $R_n \tend g$ almost everywhere. Since the sequence $(R_n)$ converges almost everywhere to $f$ up to an extraction, uniqueness of the pointwise limit gives $f=g$.
\end{proof}

Finally, we will use the following result, easy in the context of Hilbert spaces, which allows to deduce strong convergence of a sequence from weak convergence and convergence of the norms.

\begin{prop}[ See  {\cite[Lemma 3.1.6.]{Zheng}}]\label{p:NormConvergence}
    Consider a uniformly convex Banach space $X$ and a sequence $(f_n)$ of $X$ functions such that we have
    \begin{equation*}
        f_n \wtend f \qquad \text{in } X
    \end{equation*}
    for some $f \in X$ and
    \begin{equation*}
        \| f \|_{X} \tend \| f \|_{X}.
    \end{equation*}
    Then the sequence converges in the norm topology of $X$.
\end{prop}


\newpage

\addcontentsline{toc}{section}{References}
{\small
\begin{thebibliography}{xxx}



\bibitem{AbbatielloFeireisl} A. Abbatiello, and E. Feireisl:
\textit{On a class of generalized solutions to equations describing incompressible viscous fluids}.
Annali di Matematica Pura ed Applicata (1923-), 199, 1183-1195 (2020).

\bibitem{Amann} H. Amann:
\textit{Stability of the rest state of a viscous incompressible fluid}.
Archive for rational mechanics and analysis (1994), 126, 231-242.

\bibitem{BCD} H. Bahouri, J.-Y. Chemin and R. Danchin:
\textit{``Fourier analysis and nonlinear partial differential equations''}.
Grundlehren der Mathematischen Wissenschaften (Fundamental Principles of Mathematical Sciences), Springer, Heidelberg, 2011.

\bibitem{BerselliDieningRuzicka} L. Berselli , L. Diening, and M. R\r{u}\u{z}i\u{c}ka:
\textit{Existence of strong solutions for incompressible fluids with shear dependent viscosities}.
Journal of Mathematical Fluid Mechanics (2010), 12, 101-132.

\bibitem{BerselliRuzicka1} L. Berselli, and M. R\r{u}\u{z}i\u{c}ka:
\textit{Global regularity for systems with $p$-structure depending on the symmetric gradient}. Advances in Nonlinear Analysis, 9(1), (2018) 176-192.

\bibitem{BerselliRuzicka2} L. Berselli, and M. R\r{u}\u{z}i\u{c}ka:
\textit{Natural second-order regularity for parabolic systems with operators having $(p, \delta)$-structure and depending only on the symmetric gradient}. 
Calculus of Variations and Partial Differential Equations, (2022) 61(4), 137.

\bibitem{BerselliRuzicka3} Berselli, and M. R\r{u}\u{z}i\u{c}ka:
\textit{Global regularity properties of steady shear thinning flows}.
Journal of Mathematical Analysis and Applications, (2017) 450(2), 839-871.

\bibitem{Brezis} H. Brezis:
\textit{Functional analysis, Sobolev spaces and partial differential equations}.
Universitext Springer, New York, 2011. xiv+599 pp.

\bibitem{BGMS} M. Bul\'{i}\u{c}ek, P. Gwiazda, J. M\'{a}lek, and A. \'{S}wierczewska-Gwiazda:
\textit{On unsteady flows of implicitly constituted incompressible fluids}.
SIAM Journal on Mathematical Analysis (2012), 44(4), 2756-2801.

\bibitem{BurczakModenaSzekelyhidi} J. Burczak, S. Modena, and L. Sz\'{e}kelyhidi:
\textit{Non uniqueness of power-law flows}.
Communications in Mathematical Physics (2021), 388, 199-243.

\bibitem{CaoChenQianShaYu} H. Cao, Y. Chen, K. Qian, X. Sha and K. Yu:
\textit{Shear-thickening behavior of modified silica nanoparticles in polyethylene glycol}. 
Journal of Nanoparticle Research, (2012) 14, 1-9.

\bibitem{CCL} A. Castro, D. C\'ordoba and D. Lear:
\textit{Global existence of quasi-stratified solutions for the confined IPM equation}.
Arch. Ration. Mech. Anal.232(2019), no.1, 437--471.

\bibitem{CCCGW} D. Chae, P. Constantin, D. Córdoba, F. Gancedo and J. Wu: 
\textit{Generalized surface quasi-geostrophic equations with singular velocities}. Comm. Pure Appl. Math. 65 (2012), n. 8, pp. 1037?1066.

\bibitem{CCL2023} L. Chupin, N. Cîndea and G. Lacour:
\textit{Variational inequality solutions and finite stopping time for a class of shear-thinning flows}.
\url{https://doi.org/10.48550/arXiv.2112.02871}

\bibitem{CindeaLacour} N. C\^{i}ndea, and G. Lacour:
\textit{Null controllability of quasilinear parabolic equations with gradient dependent coefficients}. 
arXiv preprint arXiv:2304.08022. (2023).
\url{https://doi.org/10.48550/arXiv.2304.08022}

\bibitem{Cobb} D. Cobb:
\textit{On the Well-Posedness of a Fractional Stokes-Transport System}.
\url{https://doi.org/10.48550/arXiv.2301.10511}

\bibitem{CF1} D. Cobb and F. Fanelli:
\textit{Rigorous derivation and well-posedness of a quasi-homogeneous ideal MHD system}.
Nonlinear Anal. Real World Appl.60(2021), Paper No. 103284, 36 pp.

\bibitem{ColomboTione} M. Colombo, and R. Tione:
\textit{Non-classical solutions of the $p$-Laplace equation}.
arXiv preprint arXiv:2201.07484 (2022).
\url{https://doi.org/10.48550/arXiv.2201.07484}

\bibitem{DGL} A.-L. Dalibard, J. Guillod and A. Leblond:
\textit{Long-time behavior of the Stokes-transport system in a channel}.
\url{https://doi.org/10.48550/arXiv.2306.00780}

\bibitem{Danchin2023} Raphaël Danchin:
\textit{Global Well-Posedness for 2D Inhomogeneous Viscous Flows With Rough Data Via Dynamic Interpolation}.
ffhal-04227173f (2023).



\bibitem{DiBenedetto} E. DiBenedetto:
\textit{Degenerate parabolic equations}. 
Springer Science \& Business Media (1993).

\bibitem{DieningRuzickaWolf} L. Diening, M. R\r{u}\u{z}i\u{c}ka, and J. Wolf:
\textit{Existence of weak solutions for unsteady motions of generalized Newtonian fluids}. Annali della Scuola Normale Superiore di Pisa-Classe di Scienze (2010), 9(1), 1-46.


\bibitem{DpL} R. J. DiPerna and P.-L. Lions:
\textit{Ordinary differential equations, transport theory and Sobolev spaces}.
Invent. Math. 98 (1989), no.3, pp. 511?547.

\bibitem{Doob} J. Doob: 
\textit{Measure theory}.
Springer Science \& Business Media (1993), Vol. 143.

\bibitem{Elgindi17} T. Elgindi:
\textit{On the asymptotic stability of stationary solutions of the inviscid incompressible porous medium equation}.
Arch. Ration. Mech. Anal.225(2017), no.2, 573--599.

\bibitem{evans} L.C. Evans :
\textit{Partial differential equations}.
2nd Edition, Graduate Studies in Mathematics, Vol. 19, American Mathematical Society, 2010,

\bibitem{FeireislHillairetNecasova} E. Feireisl, M. Hillairet, and \u{S}. Ne\u{c}asov\'{a}:
\textit{On the motion of several rigid bodies in an incompressible non-Newtonian fluid}. Nonlinearity (2008), 21(6), 1349.

\bibitem{Feuillebois} F. Feuillebois: 
\textit{Sedimentation in a dispersion with vertical inhomogeneities}. 
J. Fluid Mech. 139, 145--171 (1984).

\bibitem{Folland} G. Folland: 
\textit{Real analysis: modern techniques and their applications}.
John Wiley \& Sons (1999), Vol. 40..

\bibitem{FGSV} S. Friedlander, F. Gancedo, W. Sun and V. Vicol: 
\textit{On a singular incompressible porous media equation}. 
J. Math. Phys. 53 (2012), n. 11, 115602, 20 pp.

\bibitem{FrehseRuzicka} J. Frehse, and M. R\r{u}\u{z}i\u{c}ka:
\textit{Non-homogeneous generalized Newtonian fluids}.
Mathematische Zeitschrift (2008), 260(2), 355-375.

\bibitem{Grafakos} L. Grafakos:
\textit{Modern fourier analysis}.
New York: Springer (2009), Vol. 250, pp. xvi+-504.

\bibitem{Grayer2023} H. Grayer:
\textit{Dynamics of Density Patches in Infinite Prandtl Number Convection}. 
Arch Rational Mech Anal 247, 69 (2023).

\bibitem{GSW} P. Gwiazda, A \'{S}wierczewska-Gwiazda, and A.  Wr\'{o}blewska:
\textit{Generalized Stokes system in Orlicz spaces}.
Discrete Contin. Dyn. Syst (2012), 32(6), 2125-2146.

\bibitem{Hofer2018} R. Höfer:
\textit{Sedimentation of inertialess particles in Stokes flows}.
Comm. Math. Phys. 360 (2018), no. 1, 55--101.

\bibitem{HS} R. Höfer and R. Schubert:
\textit{The influence of Einstein's effective viscosity on sedimentation at very small particle volume fraction}.
Ann. Inst. H. Poincaré C Anal. Non Linéaire 38 (2021), no. 6, 1897--1927.


\bibitem{Inversi} M. Inversi:
\textit{Lagrangian solutions to the transport-Stokes system}.
Nonlinear Anal.235(2023), Paper No. 113333, 29 pp.

\bibitem{JLL} J.-L. Lions:
\textit{Quelques méthodes de résolution des problèmes aux limites non linéaires}.
Dunod, Gauthier-Villars Paris (1969).

\bibitem{KamenevaRobertsonSequeira} M.V. Kameneva, A.M. Robertson and A. Sequeira:
\textit{Hemorheology}.
In: Hemodynamical Flows. Oberwolfach Seminars, vol 37. Birkhäuser Basel (2008).

\bibitem{Ladyzhenskaya} O. Ladyzhenskaya:
\textit{New equations for the description of the motions of viscous incompressible fluids, and global solvability for their boundary value problems}.
Trudy Matematicheskogo Instituta Imeni VA Steklova (1967), 102, 85-104.


\bibitem{Leblond} A. Leblond:
\textit{Well-posedness of the Stokes-transport system in bounded domains and in the infinite strip}.
J. Math. Pures Appl. (9)158(2022), 120--143.

\bibitem{CardiovascularMath} F. Luca, A. Quarteroni, and A. Veneziani (eds): 
\textit{Cardiovascular Mathematics: Modeling and simulation of the circulatory system. Vol. 1}. Springer Science \& Business Media, 2010.

\bibitem{MalekNecasRuzicka} J. M\'{a}lek, J. Ne\u{c}as, and M. R\r{u}\u{z}i\u{c}ka:
\textit{On weak solutions to a class of non-Newtonian incompressible fluids in bounded three-dimensional domains: The case p>= 2}.
Advances in Differential Equations (2001), 6(3), 257-302.

\bibitem{Marchand} F. Marchand: 
\textit{Existence and regularity of weak solutions to the quasi-geostrophic equations in the spaces} $L^p$ \textit{or} $\dot{H}^{-1/2}$. 
Comm. Math. Phys. 277 (2008), n. 1, pp. 45--67.

\bibitem{Mecherbet2019} A. Mecherbet:
\textit{Sedimentation of particles in Stokes flow}.
Kinet. Relat. Models 12 (2019), no. 5, 995--1044.

\bibitem{Mecherbet2020} A. Mecherbet. 
\textit{On the sedimentation of a droplet in Stokes flow}. 
2020 To appear in Comm Math Sci.

\bibitem{MS} A. Mecherbet and F. Sueur:
\textit{A few remarks on the transport-Stokes system}.
\url{https://doi.org/10.48550/arXiv.2209.11637}

\bibitem{MichaudSoutrenon} V . Michaud and M. Soutrenon: 
\textit{Energy dissipation in concentrated monodisperse colloidal suspensions of silica particles in polyethylene glycol}. 
Colloid and Polymer Science, 292, (2014) 3291-3299.

\bibitem{Resnick} S. Resnick: 
\textit{Dynamical problems in nonlinear advective partial differential equations}. 
Ph.D. thesis, University of Chicago, 1995.

\bibitem{Roubicek} T. Roub\'{i}\u{c}ek (2013). 
\textit{Nonlinear partial differential equations with applications}.
Springer Science \& Business Media (2013), Vol. 153.

\bibitem{RudyakTretiakov} V. Rudyak and D. Tretiakov:
\textit{Viscosity and rheology of the ethylene glycol
based nanofluids with single-walled carbon
nanotubes}.
J. Phys.: Conf. Ser. 1382 012100 (2019).

\bibitem{Starovoitov} V. Starovoitov:
\textit{Behavior of a rigid body in an incompressible viscous fluid near a boundary}.
In Free Boundary Problems: Theory and Applications (pp. 313-327). Basel: Birkhäuser Basel (2003).


\bibitem{truesdell} C. Truesdell:
\textit{A first course in rational continuum mechanics. Vol. 1}.
Boston, MA: Academic Press, 1991.

\bibitem{Zheng} S. Zheng:
\textit{Nonlinear evolution equations}.
CRC Press (2004).

\end{thebibliography}
}
\end{document}